\documentclass[12pt]{amsart}

\usepackage{amsmath}
\usepackage{amssymb} 
\usepackage{array}
\usepackage{enumitem}
\usepackage{hyperref}
\usepackage{verbatim} 
\usepackage{color}
\definecolor{battleshipgrey}{rgb}{0.52, 0.52, 0.51} 
\hypersetup{
       colorlinks=true,       
    linkcolor=battleshipgrey,     
    citecolor=battleshipgrey,        
    filecolor=magenta,      
    urlcolor=battleshipgrey          
}
\usepackage{pstricks-add}

\theoremstyle{plain}
\newtheorem{theorem}{Theorem}[section]
\newtheorem{lemma}[theorem]{Lemma}

\newtheorem{definition}[theorem]{Definition}
\theoremstyle{remark}
\newtheorem{remark}{Remark}[section]
\newtheorem{example}{Example}[section]
\newtheorem*{notation}{Notation}
\newtheorem*{acknowledgment}{Acknowledgment}

\numberwithin{equation}{section}

\newcommand{\bA}{\mathbb{A}}

\newcommand{\K}{\mathbb{K}}

\newcommand{\R}{\mathbb{R}}

\newcommand{\N}{\mathbb{N}}

\newcommand{\cP}{{\mathcal P}}

\newcommand{\sss}{\mathbf{s}}

\newcommand{\ttt}{\mathbf{t}}

\newcommand{\vvv}{\mathbf{v}}
\newcommand{\www}{\mathbf{w}}

\newcommand{\eff}{\mathbf{f}}
\newcommand{\bff}{\mathbf{f}}
\newcommand{\bfg}{\mathbf{g}}

\newcommand{\bfU}{\mathbf{U}}

\newcommand{\sfQ}{\mathsf{Q}} 
\newcommand{\sfP}{\mathsf{P}} 
\newcommand{\sfG}{\mathsf{G}} 
\newcommand{\SfG}{{\mathsf{Gfi}}} 
\newcommand{\sfGfi}{{\mathsf{Gfi}}} 
\newcommand{\sfGsy}{\mathsf{Gsy}} 
\newcommand{\sfGsi}{\mathsf{Gsi}} 

\newcommand{\sfT}{\mathsf{T}}
\newcommand{\pG}{\mathsf{PG}}
\newcommand{\sfA}{\mathsf{A}} 

\newcommand{\sfn}{\mathsf{n}}
\newcommand{\sfk}{\mathsf{k}}

\newcommand{\sfone}{\mathsf{1}}
\newcommand{\sftwo}{\mathsf{2}}
\newcommand{\sfthree}{\mathsf{3}}
\newcommand{\sffour}{\mathsf{4}}

\newcommand{\sym}{\mathrm{sym}}

\newcommand{\id}{\mathrm{id}}

\newcommand{\sett}[1]{{\{ #1 \}}}
\newcommand{\settt}[1]{{\overline{\{ #1 \}}}}

\newcommand{\Bigsetof}[2]{\begin{Bmatrix} #1 \,\Big|\, #2 \end{Bmatrix}}

\newcommand{\msk}{\medskip}
\newcommand{\ssk}{\smallskip}
\newcommand{\nin}{\noindent}

\newcommand{\ul}{\underline}
\newcommand{\ol}{\overline}

\begin{document}

\title[Conceptual Differential Calculus. II]{Conceptual Differential Calculus
\\
Part II: Cubic higher order calculus}

\author{Wolfgang Bertram}

\address{Institut \'{E}lie Cartan de Lorraine \\
Universit\'{e} de Lorraine at Nancy, CNRS, INRIA \\
Boulevard des Aiguillettes, B.P. 239 \\
F-54506 Vand\oe{}uvre-l\`{e}s-Nancy, France}

\email{\url{wolfgang.bertram@univ-lorraine.fr}}



\begin{abstract} 
Following the programme set out in 
\href{http://arxiv.org/abs/1503.04623}{Part I} of this work, we develop a {\em conceptual higher order differential calculus}.
The ``local linear algebra'' defined in Part I is generalized by ``higher order local linear algebra''.
The underlying combinatorial object of such higher algebra is
 the {\em natural $n$-dimensional hypercube}, and so we qualify this calculus as ``cubic''. More precisely, we define two
 versions of conceptual cubic calculus: ``full'' and ``symmetric cubic''.   
The theory thus initiated  
 sheds new light on several foundational issues.
  \end{abstract}

\maketitle

\setcounter{tocdepth}{1}

\tableofcontents

\section{Introduction}

\subsection{``Die Gesetze des Endlichen''}
In preceding work  (\cite{Be15, Be08}), I quoted the following phrase by G.W.\ Leibniz:
``Die Gesetze des Endlichen gelten im Unendlichen weiter''
(``the rules of the finite continue to hold in the infinitesimal''), but I should admit that I really started to understand its meaning while
working on the results to be presented here. When
teaching differential calculus, we usually hasten to take the limit --  ``the infinitesimal'' --  and don't ask ourselves what
``the rules of the finite''  are.
Implicitly, we take for granted that these rules are kind of trivial and known by everybody.

\ssk
In the end, it may well appear that these ``rules of the finite'' look trivial; but they are certainly not known by everybody. 
They are trivial in the sense that they ``just'' deal with iterated directed products  --
but  to say that, we need the subtle concept of $n$-fold groupoid. I guess that such  structures
underly huge parts of mathematics. 
Some of the following is still tentative, and {\em notation} is an important issue: 
I am not yet sure to have found the best one. For these and other reasons,
 the present version of this work will not be submitted for publication in a mathematical journal. 

\ssk
I will start by introducing three different kinds of higher order calculus, then explain  what their intrinsic, or conceptual, version is,
and finally say some words  about the relation between ``the finite'' and ``the infinite''.
Let me stress once again  that this is far from being a complete theory, but rather a first picture of the tip of an iceberg 
whose shape is still unknown, and the present work would have reached its goal if the reader felt that something interesting and new
is going on here. 

\subsection{The threefold way of higher order calculus}\label{sec:threefold}
The {\em  first order conceptual calculus} developed in \cite{Be15} (quoted as Part I) is based on a groupoid interpretation of the {\em first order difference
quotient map}, or {\em slope}, of a map $f:U \to W$,
\begin{equation} \label{eqn:01}
f^{[1]} (x,v,t) = \frac{f(x+tv) - f(x)}{t}.
\end{equation}
Invertible scalars $t$ represent ``das Endliche'' and non-invertible scalars $t$ represent ``das Unendliche'' (with $t=0$
as most singular value). 
Recall that, in a topological context (cf.\ Part I),
 $f$ is called {\em of class $C^1$} if  the slope $f^{[1]}$ extends to a {\em continuous map} (still denoted by $f^{[1]}$), for the value $t=0$ included.
At higher order, there are three different ways to iterate this procedure (Section \ref{sec:threefold}): 

\begin{itemize}
\item
 {\em symmetric cubic calculus} (closest to usual calculus),
\item
 {\em simplicial calculus} (closer to the spirit of algebraic geometry),
\item
{\em full cubic calculus} (which includes both preceding approaches).
\end{itemize}

\ssk \nin
Briefly, {\em full cubic calculus}, as developed in \cite{BGN04}, consists of iterating the procedure in the ``brutal way'', that is, by
looking at higher order slopes $f^{[n]}$ defined by $f^{[2]} := (f^{[1]})^{[1]}$, and so on. These maps
are very hard to understand, so let's
 propose 
a ``tamed''  version of the iteration procedure: fix a scalar $t$ and consider the map 
\begin{equation}
f^{[1]}_t: (x,v) \mapsto f^{[1]} (x,v,t) ,
\end{equation}
Iterating this, we  define $f^{[2]}_{t_1,t_2} := (f_{t_1}^{[1]})_{t_2}^{[1]}$, and so on: 
at each step, we fix a scalar $t_k$, take  the slope with this fixed value of scalar, and then look what happens when  the
multi-scalar $(t_1,\ldots, t_n)$ tends to zero (Section \ref{subsec:symmetric}).  In fact, 
this amounts to consider  a sub-structure of the ``full'' $f^{[n]}$, which turns out to
contain already ``most'' of the information of the full map, and  having the advantage that it is much easier understandable
(Theorem \ref{th:symmetricformula}).
In particular, there is a natural action of the symmetric group $S_n$ preserving the structure (``Schwarz' theorem''), 
whence the term {\em symmetric} cubic calculus. 
However,  for proving the general Taylor formula from \cite{BGN04}, 
the symmetric cubic setting is not sufficient, 
 and for this reason we have developed in \cite{Be13} a
{\em simplicial differential calculus}.
  It should be possible and important (in particular, in order to relate the present work to approaches 
used in algebraic geometry) to develop a {\em conceptual and categorical
formulation of simplicial calculus} -- however,  to keep the present work in reasonable bounds,  this is left for subsequent work.

\subsection{Conceptual version: the three symbols $\sfG^\sfn$, $\sfGsy^\sfn$, $\sfGsi^\sfn$}
I have chosen the sans serif letter $\sfG$, reminding the term {\em groupoid},  combined with 
$\sfn$ standing for the set $\sfn = \{ 1,\ldots, n \}$, to denote  certain {\em $n$-fold groupoids},
which are the conceptual objects corresponding to the three calculi.
More precisely, 

\begin{itemize}
\item
$\sfG^\sfn U$ is the {\em full cubic $n$-fold groupoid} of the domain $U$,
\item
$\sfGsy^\sfn U$ is the {\em symmetric $n$-fold groupoid} of $U$,
\item
 $\sfGsi^\sfn U$ will denote a simplicial (pre)groupoid version of the preceding. 
\end{itemize}

\nin 
Before going ahead, I should reassure the reader who is not a specialist in category theory and who never before met  $n$-fold groupoids
and categories: I have been in the same situation, and so the present text is also designed to introduce a newcomer to these gadgets.
 In this respect, I hope that the presentation given
here (Appendices \ref{sec:hypercube} and \ref{App:nfoldCats}) may be useful: 
{\em small $n$-fold categories} and {\em groupoids} are presented as ``algebraic structures'' in the usual sense, defined by sets, structure maps
and identities, so a small $n$-fold category is given by 

\begin{itemize}
\item
a familiy of sets $C_\alpha$, one for each {\em vertex} $\alpha$ of the natural hypercube 
$\cP(\sfn)$,
\item
a family of small categories, one for each {\em edge} $(\beta,\alpha)$ of the hypercube $\cP(\sfn)$,
\item
such that for each {\em face} of the hypercube these data define a small double category (as defined in Part I, 
appendix C).\footnote{ After having written this text, 
I found an essentially equivalent presentation, designed for category theorists,
 in Section 2 of  \cite{FP10}. See also \cite{GM12}, where a similar description is
given  for {\em $n$-fold vector bundles}, using a quite fancy
terminology  of ``hops'' and ``runs'',
cf.\ Th.\  \ref{th:tangentbundle}.}
\end{itemize}

\nin
Seen that way, small $n$-fold categories are computable:  for $\sfGsy^\sfn U$ formulae are
explicit and fairly simple (Theorem \ref{th:Sym2}), whereas for $\sfG^\sfn U$ they are far more complicated 
 -- we make some effort to present them in algorithmic form, that could be programmed  on a machine 
(Theorems
 \ref{th:fullcompute},  \ref{th:fullcompute1},  \ref{th:fullcompute2}).

\ssk
Having said this, I hope the general reader is willing to accept that the iteration procedure, leading from first order to higher order conceptual
calculus, is entirely ``canonical'', that is, following from general principles; once the concepts are clear, most proofs are ``by straightforward  induction'':
for $n=1$, all three groupoids coincide with the groupoid $U^\sett{1}$ defined in Part I, so,
if $\sfQ$ is one of the symbols $\sfG,\sfGsy$, then
$\sfQ^\sfone U = U^\sett{1}$. 
For $n>1$, to compute structure data of $\sfQ^\sfn U$, we just have to apply a copy of $\sfQ^\sfone$ to the data of 
$\sfQ^{\mathsf{n-1}} U$.
However, since the object set of $\sfG^\sfone U$ is a cartesian product
$U \times \K$, we need to know how $\sfQ$ behaves with respect to cartesian
products and general pullbacks. 
At this stage, the difference between the three calculi comes in:  the key notion to understand this is the one of {\em scaleoid}.

\subsection{Scaleoids, pullbacks and cartesian products}
By  ``terminal object'', 
we mean the zero-subset $\{ 0 \}$ of the zero $\K$-module $\{ 0 \}$; we just denote it by $0$. 
The $n$-fold groupoid $\sfQ^\sfn 0$ is called the 
{\it $n$-th order scaleoid}. More precisely:

\begin{itemize}
\item
in symmetric calculus, $\sfGsy^\sfn 0 \cong \K^n$, and since variables $t_k$ are kept constant at each iteration step, the groupoid structure
here is the trivial one,
\item
in full calculus, the top vertex of $\sfG^\sfn 0$ is $\K^{2^n-1}$, and the $n$-fold groupoid structure is non-trivial and rather complicated.
\end{itemize}

\nin
To relate scaleoids with cartesian products, a basic property of our symbols $\sfQ^\sfn$ is that {\em they are compatible with pullbacks}
(see Appendix \ref{app:pullback}): they satisfy the rule
\begin{equation}
\sfQ^\sfn (A \times_C B) = \sfQ^\sfn A \times_{\sfQ^\sfn C} \sfG^\sfn B.
\end{equation}
The usual cartesian product $A\times B$ is the pullback over a point $A \times_0 B$, whence
\begin{equation}
\sfQ^\sfn (A \times  B) = \sfQ^\sfn A \times_{\sfQ^\sfn 0} \sfG^\sfn B.
\end{equation}
Now, in symmetric calculus, the groupoid $\sfGsy^\sfn 0$ is $\K^n$ with trivial structure, 
so we can fix scalars,
get an $n$-fold groupoid $\sfGsy^\sfn_\ttt U$ for fixed $\ttt$
 and use the preceding relation in the form
\begin{equation}
\sfGsy_\ttt^\sfn (A \times B) =  \sfGsy_\ttt^\sfn  A \times \sfGsy_\ttt^\sfn B ,
\end{equation}
expressing that  $\sfGsy^\sfn_\ttt$ is a {\em product preserving functor}, which is a 
well-established concept in differential geometry (see \cite{KMS93}). 
In particular,  this relation implies that the philosophy of {\em scalar extension}, known
from algebraic geometry  and used in our preceding work \cite{Be08, BeS14, Be14}, wholely applies in the present context. 
In full cubic calculus, since $\sfG^n 0$ is highly non-trivial, computations cannot be untied so easily, and their description, though entirely
explicit, is a bit tedious (Theorem  \ref{th:fullcompute}).
Certainly, the abstract algebraic structure of the full scaleoid $\sfG^n 0$ deserves a more profound study in its own right.  
 
\subsection{Iterated pair groupoids, and the rules of the finite}
We consider the $n$-fold groupoid $\sfG^\sfn U$ as an {\em $n$-fold magnification of the ``space'' $U$}. It is a kind of $n$-th
order version of {\em Connes' tangent groupoid} (cf.\ Part I, Theorem 2.7), realizing an interpolation between the
$n$-th order tangent bundle and the $n$-fold pair groupoid: recall that
applying the natural iteration procedure to the pair groupoid $M \times M$ over a set $M$, one gets the
{\em $n$-fold pair groupoid} $\pG^\sfn M$  (Appendix \ref{app:pg}).
For $\ttt = (1,\ldots,1)$, 
and more generally for $\ttt = (t_1,\ldots,t_n)$ such that each $t_i$ is invertible, we have a natural isomorphism of 
$n$-fold groupoids
\begin{equation}
\sfGsy_\ttt ^\sfn U \cong \pG^\sfn U.
\end{equation}
In full calculus, the situation is similar but more subtle  (Theorem \ref{th:Full1}).
Thus one may say 
that  ``the $n$-fold pair groupoid governs the rules of the finite''.
On the ``infinitesimal side'',
for $t = (0,\ldots,0)$, we get the {\em iterated tangent bundle} 
\begin{equation}
\sfGsy_{(0,\ldots,0)}^\sfn U = \sfT^\sfn U 
\end{equation}
which is not only an $n$-fold groupoid, but an ``$n$-fold group bundle'' since all target and source maps coincide for $\ttt = 0$
(Theorem \ref{th:tangentbundle}).

\subsection{Symmetry}
The general cubic $n$-fold groupoid $\sfG^\sfn U$ is {\em not edge-symmetric} (in contrast to those considered by 
Brown and Higgins \cite{BH81}). It is a non-trivial fact that $\sfGsy^\sfn U$ {\em is} edge-symmetric (Theorem \ref{th:Sym2}),
and this symmetry can be interpreted by saying that the
 ``rule of the finite'',  stating that the $n$-fold pair groupoid is edge-symmetric, continues to hold for all parameters $\ttt$.
By the way,  this is the most natural proof of {\em
  Schwarz' theorem} on the symmetry of second  differentials, as given in topological differential calculus
  (\cite{BGN04, Be08}).  
In full calculus, the ``rules of the finite'' are more complicated:
there ought to be something replacing Schwarz' theorem -- what it is, is not yet clear to me.

\subsection{Homogeneity} 
The next  ``rule of the finite'' is
{\em homogeneity under scalars}. For $n=1$, it has been formulated conceptually  as a 
{\em structure of double category $U^\settt{1}$} (Part I).
By the general iteration procedure,  in full cubic calculus this leads to define a {\em small $2n$-fold category} that we
denote by $\sfG^{\ol \sfn} U$
(the letter $\sfG$ is used although this is not a groupoid). 
Everything said so far about $\sfG^n$ applies here, {\em mutatis mutandis}, and the structure is, of course, even more complicated.
However,
since homogeneity makes no sense for a fixed scalar $t$, this construction cannot be implemented into symmetric calculus, and
we have to add homogeneity ``by hand'' as a property that can be respected (or not) by the morphisms of $\sfGsy^\sfn U$.

\subsection{$C^n$- and $C^\infty$-laws}
{\em Laws of class $C^n$ over $\K$} are now defined in a canonical way: 
they are  morphisms of $k$-fold groupoids for $k=0,\ldots,n$, each compatible with the preceding, and likewise for $n=\infty$.
These concepts can be defined with respect to each of the three symbols $\sfG, \sfGsy, \sfGsi$.
For $\K = \R$, 
 or other topological base rings, and domains $U$ open in topological $\K$-modules $V$,
 smooth maps 
   give rise to such laws, in a unique way (Theorem \ref{th:topCn}). 
We also prove that {\em polynomial laws} (in the sense of Roby, cf.\ Part I) give rise to laws of class $C^\infty$ (Theorem \ref{th:polylaws}).
The definition of {\em $C^\infty$-manifold laws} follows the pattern given in Part I
(see Section \ref{sec:manifolds}).

\ssk
Every law has a  ``finite part''  which respects the ``rules of the finite'', since it is uniquely determined by its basic set-map
$f$ (Theorem \ref{th:finitelaws}); however, the law need not
 respect all ``rules of the infinitesimal'';  if it respects scalar actions, we call it
{\em homogeneous}; if it respects symmetry, we call it {\em symmetric}. 
According to ``Leibniz principle'', the laws induced by smooth maps in topological calculus do respect all  rules.
At a first glance, it might appear to be a drawback of our approach that such properties {\em are not automatic} for abstract
laws; however, in view of a deeper understanding, it rather opens  the very interesting question of ``classifying'' such
rules and of understanding ``how many independent ones'' of them exist, and it points the way towards calculi
having different rules than the classical ones, such as {\em supercalculus}.

\subsection{Towards geometry: the neighbor relation}
In the present work, I concentrate on developing the {\em formalism}: the $n$-fold groupoid $\sfG^\sfn M$ and $2n$-fold small cate\-gory
$\sfG^{\ol \sfn} M$ are canonical and chart-independent objects coming with any  ``smooth space'' $M$, and they contain all information 
about its ``smooth structure''. Once the formal theory is developed, we may  ask for its ``geometric interpretation''. Such geometric interpretation
will be more ``synthetic'' than the usual, rather analytic, differential geometry, since we directly work with invariant and chart-independent 
structures. Indeed, we shall be able to use, in our framework,
 the language of {\em synthetic differential geometry} (SDG, see \cite{Ko06, Ko10, MR91}) --  indeed, I believe that the
present approach yields  a new and very interesting model of SDG.  
Such topics will be treated in subsequent work. However, to understand the present text, it may help the reader to have 
 in  mind some kind of ``geometric picture'', and therefore I shall already here say some words on topics that will be taken up in
later work.
 
\ssk
Let us call ``primitive points'' the elements of the
 point set $M$ underlying our smooth space, and think of  the $n$-fold groupoids $\sfG^\sfn M$ as representing
higher ``stages'' or ``levels of definition'' of $M$,
whence the proposition of terminology to call $\sfG^\sfn M$ the {\em $n$-th magnification} of the ``space $M$''
(for a most general theory, one will have to speak of {\em local} $n$-fold groupoids, cf.\ Part I, but this does not chage much
in the following). 
  More precisely, a point $a$ has ``stage'' or ``level $k$'' if it belongs to a vertex set
$\sfG^{\alpha;\sfn} M$ for a vertex $\alpha \in \cP(\sfn)$ of order $\vert \alpha \vert =k$; so the primitive points are of stage $0$.
Every point $a \in \sfG^\sfn M$ has a {\em scale}, which is its image $\ttt = \ttt_a$ in the scaleoid $\sfG^\sfn 0$ under the
canonical morphism $ \sfG^\sfn M \to \sfG^\sfn 0$. 
A point $a$ is called {\em finite} if its scale $\ttt_a$  belongs to the finite part $\sfGfi^\sfn 0$ of the scaleoid, and {\em infinitesimal} else.

\ssk
The {\em first order neighbor relation}, which Kock in \cite{Ko10} considers ``the main actor'' of his presentation of SDG,
can be defined as follows: 
say that two points $a,b$ of same level $k$ and belonging to a common vertex set $\sfG^{\alpha;\sfn} M$ are {\em first order neighbors},
$a \sim_1 b$, if there exists an edge $(\beta,\alpha)$ such that $a$ and $b$ have same image under the source projection of the
edge category defined by this  edge:
$\pi^{\beta,\alpha;\sfn}_0 (a) = \pi^{\beta,\alpha;\sfn}_0(b)$.
This relation is symmetric and reflexive, but not transitive since there are $k$ such edge projections!
In symmetric calculus, neighbors have the same scale; in full cubic calculus, the situation is more intricate since the 
``magnifcation procedure'' is not only applied to the space, but also to the scale itself.
Similarly, the {\em second neighbor relation}, $a \sim_2 b$, is defined by using {\em double source projections} coming from faces 
with top vertex $\alpha$, and so on. As in SDG, it follows immediately that $a \sim_k b$, $b \sim_\ell c$ implies
$a \sim_{k + \ell} c$. When $a,b,c$ are infinitesimal  (with scale essentially zero), then heuristic arguments show that these neighbor relations
coincide indeed with those used by Kock in \cite{Ko06, Ko10}. The advantage, compared to SDG,
is that the same terminology, and all arguments,
literally also apply for finite scale -- all ``synthetic reasoning'' applies in a perfectly rigourous way if we start on the ``finite side'',
by formulating things carefully in a language using natural groupoid strucures, and then letting scales ``tend to zero''.
To carry this out, on the finite side, we may think of ``points of level $k$'' in different ways: with regard to points of higher level, they are
 just ``points'' (objects); but with respect to points of lower level, they rather are ``arrows'' or ``segments'', namely, the ``point $a$''  is interpreted
 as the
segment or arrow from its source $x=\pi_0(a)$ to its target $y=\pi_1(a)$. If targets and sources match, such segments can be composed: this
operation is natural, and it is preseved by all smooth maps. 
If targets and sources don't match, there is no natural composition; however, there are unnatural compositions: this is 
what {\em connection theory} is about (subsequent work). 
The groupoid aspect of connection theory, going back to Charles Ehresmann (\cite{E65, KPRW07}),
 is strongly stressed in SDG (\cite{Ko10}), but I have the impression that only with the present approach
it really becomes clear why it is so natural:   the very foundations of calculus rely on it.
I believe that many other topics of local differential geometry will be amenable, in a similar way, to a ``synthetic'' approach
giving a link between ``the finite'' and ``the infinitesimal''.

\begin{notation}
We use small sans serif letters to denote the following finite subsets of $\N$:
$$
\sfone = \{ 1 \}, \quad \sftwo = \{ 1,2 \}, \quad \sfthree = \{ 1,2,3 \},\ldots ,
{\mathsf n}=\{ 1,\ldots, n \} .
$$
We let also $\sfn' = \{ 1', \ldots, n'\}$ be a formal copy of $\sfn$, disjoint from $\sfn$, and
$$
\ol \sfn := \sfn \cup \sfn' = \{ 1,1', 2,2', \ldots, n,n' \} .
$$
The power set of a set $A$ is denoted by $\cP(A)$. In particular, $\cP(\sfn)$ is the {\em natural $n$-hypercube}, 
and $\cP(\ol \sfn)$ is the natural $2n$-hypercube (Appendix \ref{sec:hypercube}). 
The natural $n$-hypercube is the index set for $n$-fold categories (Appendix \ref{App:nfoldCats}).
To simplify notation, for elements of a $\K$-module $V$, instead of $v_\emptyset, v_{ \{ 1 \} }, v_{\{ 1,2 \}},\ldots$ we throughout write
$v_0,v_1,v_{12},\ldots$, and likewise for elements of $\K$.
\end{notation}

\begin{acknowledgment}
I  gratefully acknowledge that I have a permanent position at a public university.
I also acknowledge having received the advice that  long range and
novel work is not suitable for mathematical journals, and that one should better wait
and collect material until it is presentable in monograph form --
I shall follow this advice (and
I just hope that the delay will not  be longer than the usual length of a human life). 
Anyways, comments and suggestions are welcome, and I thank the reader in advance for letting me know.
\end{acknowledgment}

 \section{The threefold way of higher order calculus}\label{sec:threefold}

\subsection{Full cubic calculus}\label{subsec:full}
Recall from Part I the definitions of $U^{[1]}$, $(U^{[1]})^\times$, and $f^{[1]}$.
The ``full'' cubic calculus (\cite{BGN04}) is based on direct iteration of (\ref{eqn:01}).

\begin{definition}[Full cubic $C^n$] \label{def:full}
Under the assumptions of {\em topological differential calculus} (i.e., $\K$ a topological base ring with dense unit group, $V,W$ topological
$\K$-modules, $U$ open in $V$; cf.\ Part I), a map $f:U \to W$ is called {\em of class $C^n$} if it is $C^{n-1}$ and if 
$f^{[n-1]}$ is $C^1$; so  the map  $f^{[n]} := (f^{[n-1]})^{[1]}$ exists and is continuous on its domain of definition
$$
U^{[n]} = (U^{[n-1]})^{[1]}.
$$
\end{definition}

\nin 
The {\em full $n$-th order difference quotient map} $f^{[n]}$ has a ``cubic structure'': it is a  combination of values of $f$ at $2^n$
{\em evaluation points}.  We use ``cubic notation'': for $n=1$, instead of $(x,v,t)$, we write $(v_0,v_1,t_1)$, and 
 $f^{[2]} = (f^{[1]})^{[1]}$ is  given by
 
\ssk 
$ \quad f^{[2]} \bigl( ( v_0,v_1,t_1), (v_2,v_{12},t_{12}) , t_2 \bigr) $ 
\begin{align}\label{eqn:02}
\qquad \qquad  &=  \frac{1}{t_2} \Bigl(    f^{[1]} \bigl((v_0,v_1,t_1) + t_2 
(v_2,v_{12},t_{12})\bigr) - f^{[1]}(v_0,v_1,t_1)  \Bigr)
\cr
\qquad \qquad  & =
\frac{
f(v_0 + t_2 v_2 + (t_1+ t_2 t_{12}) (v_1+ t_2 v_{12})) - f(v_0 + t_2 v_2)}{t_2(t_1 + t_2 t_{12}) }
-
\frac{f(v_0+t_1 v_1) - f(v_0)}{t_2 t_1}
\end{align}
For $n >2$, it is rather hard to give a ``closed formula'' for  $f^{[n]}$. 

\subsection{Symmetric cubic calculus}\label{subsec:symmetric}
Things are considerably simplified if, at each step, one 
``derives'' only in direction of the ``space variables'' $v$, and suppresses derivation in direction of the ``time variables'' $t$:

\begin{definition}[Symmetric $C^{n,\sym}$]\label{def:sym}
Under assumptions as above, for fixed  $t\in \K$,  let 
$U_t:=\{ (v_0,v_1) \in V^2 \mid \, v_0 \in U, v_0 + t v_1 \in U \} $ and write, for a $C^1$-map $f:U \to W$,
$$
f_t^{[1]}:  U_t \to W, \quad
(v_0,v_1)\mapsto f^{[1]}(v_0,v_1,t) .
$$
By induction, for fixed $\ttt = (t_1,\ldots,t_n) \in \K^n$, we put, whenever defined,
$$
f_\ttt^{[n]}:=
f_{t_1,\ldots,t_{n}}^{[n]}: = (f_{t_1,\ldots,t_{n-1}}^{[n-1]})_{t_{n}}^{[1]} : U_\ttt^{[n]} := ( \ldots (U_{t_1})_{t_2} \ldots)_{t_n} \to W ,
$$
and we say that 
 $f$ is  {\em $C^{n,\sym}$} if it is $C^{n-1,\sym}$, and if the map
 $f^{[n-1]}_{(t_1,\ldots,t_{n-1})}$ is $C^1$ (so that, in particular, 
 $f^{[n]}_\ttt$ is defined and continuous).
\end{definition} 

\nin
At second order, this amounts to taking $t_{12}=0$ in (\ref{eqn:02}):  
\begin{align}\label{eqn:03}
f_{t_1,t_2}^{[2]} (v_0,v_1,v_2,v_{12})  =f^{[2]} \bigl( ( v_0,v_1,t_1), (v_2,v_{12},0) , t_2 \bigr) 
\qquad \qquad   \qquad \qquad \qquad \quad
\cr
\qquad \qquad   \qquad =
\frac{
f(v_0 + t_1 v_1 + t_2 v_2 + t_1 t_2 v_{12}) - f(v_0 + t_1 v_1 ) - f(v_0 + t_2 v_2 ) + f(v_0) }{t_1 t_2}
\end{align}
This formula generalizes: using notation about hypercubes from Appendix \ref{app:hypercube},

\begin{theorem}\label{th:symmetricformula}
For invertible scalars $t_1,\ldots,t_n$ and $\vvv =  (v_\alpha)_{\alpha \in \cP(\sfn)} \in U_\ttt^{[n]}$
 with $v_\alpha \in V$, we have,
$$
f_\ttt^{[n]} (\vvv ) 
=  \frac{1}{\prod_{i=1}^n t_i} \sum_{\alpha \in \cP(\sfn)} (-1)^{\ell\alpha+1}  \cdot 
f
\bigl( \sum_{\beta \in \cP( \alpha)} t_{\beta_1}  \cdots t_{\beta_\ell}   v_\beta  \bigr) .
$$
 \end{theorem}
 
\begin{proof} By induction:
for $n=1$, $t=t_1$, $\cP(\sfone)= \{ \{ 1 \}, \emptyset \}$, we get 
$f_t^{[1]}(v_\emptyset ,v_1)= \frac{1}{t} ( f(v_\emptyset + t v_1) - f(v_\emptyset))$, as we should.
Assume the fomula holds at rank $n$. In order to compute the formula for
$f^{[n+1]}_{t_1,\ldots,t_{n+1}}$, compute $h_{t_{n+1}}^{[1]}$ for
$h = f^{[n]}_{t_1,\ldots,t_{n}}$.
Each of the $2^n$ ``old'' evaluation points
$x_\alpha := \sum_{\beta \in \cP( \alpha)} t_{\beta_1}  \cdots t_{\beta_\ell}   v_\beta $
gives rise to another ``new'' evaluation point
$x_\alpha + t_{n+1} v_{\alpha \cup \{ n+1 \} }$, which together with the  $2^n$ ``old'' evaluation points
give the $2^{n+1}$ evaluation points from the claim. 
 Mind also the sign-changes and division by $t_{n+1}$ in order to get the
formula from the claim at rank $n+1$.
\end{proof}

The formula may be qualified as {\em cubic}, and (in contrast to (\ref{eqn:02}))  has 
an obvious symmetry property with respect to the action of the permutation group $S_n$ acting on indices,
 whence the term ``symmetric''. 
Moreover, the usual higher order differentials $d^nf(v_0)$ appear when all $t_i=0$ and all $v_\alpha = 0$ whenever
$\alpha$ has more than one digit, and then the $S_n$-symmetry gives Schwarz' lemma (symmetry of higher differentials), cf.\
\cite{BGN04}. Thus the symmetric cubic calculus allows to recover easily certain notions of usual calculus.

\subsection{Simplicial calculus}
In \cite{Be13}, I  take as departure point an ``explicit formula'' which can be seen as a simplicial analog
of the explicit formula from Theorem \ref{th:symmetricformula}. 
We shall give a conceptual interpretation  in subsequent work. 

\begin{theorem}
Let $\K = \R$ or $\mathbb C$, $V$, $W$  locally convex topological vector spaces, and $f:U \to W$ a map defined on an open set of  $V$.
Then the following are equivalent:
\begin{enumerate}
\item
$f$ is of class $C^n$ in the ``usual'' sense (Michal-Bastiani, cf. \cite{BGN04}), 
\item
$f$ is $C^{n,\sym}$ in the sense defined above,
\item
$f$ is $C^n$ in the full cubic sense defined above,
\item
$f$ is $C^{<n>}$ as defined in \cite{Be13}.
\end{enumerate}
\end{theorem} 
 
 \begin{proof}
 The equivalence of (1) and (3) is shown in \cite{BGN04} (and for the implication $(1) \Rightarrow (3)$ we need the specific assumptions
 on $\K,V$, and $W$).
 It is clear that ``full $C^n$'' is stronger than ``symmetric $C^n$'', so (3) $\Rightarrow$ (2).
 By induction, one shows that
    (2) $\Rightarrow$ (1); indeed,  (2) implies existence and continuity of $f_0^{[1]}$, which amounts to
 (1) (existence and joint continuity of all directional derivatives). 
 As shown in \cite{Be13}, Th.\ 2.6., we have $(3) \Rightarrow (4)$ for any topological ring.
 Finally, 
 $(4) \Rightarrow (1)$ since by \cite{Be13}, Cor.\ 2.8, the usual differentials are partial maps of $f^{<k>}$.
 \end{proof}

\nin
If $V$ is finite-dimensional, all of these properties are equivalent  to any of the classical definitions of the class $C^n$.
For a general topological ring we have $(3) \Rightarrow (2) \Rightarrow (1)$, but the converses fail in general.
 For the converse of $(3) \Rightarrow (4)$, we do not know.

\section{The general pattern}

We define the symbols $\sfG^\sfn, \SfG^\sfn,\sfG^{\ol\sfn}$ and $\SfG^{\ol\sfn}$, and derive their basic properties.

\subsection{Zero order}
For linear sets $U \subset V$, $U' \subset V'$ and a map $f:U \to U'$, we let
$$
\sfG^\emptyset U := \sfG^{\ol \emptyset} U : = U, \qquad
\sfG^\emptyset f := \sfG^{\ol \emptyset} f : = f.
$$

\subsection{First order groupoids and double cats}\label{sec:first}
We introduce first order notation. 
Recall (Appendix \ref{App:hypercubes}) that
the {\em $1$-cube} has two vertices: $\emptyset$ and $\sfone$, corresponding to the sets of objects, resp.\ of morphisms,
and one edge:
$(\emptyset, \sfone)$, corresponding to the projections from morphisms to objects.
The {\em extended $1$-cube} has
 four vertices, and so there will be four vertex sets labelled by the four vertices
$\emptyset, \{ 1 \}, \{ 1'\}, \{ 1,1'\}$ of $\ol \sfone = \{ 1,1' \}$; there are  four pairs of projections, four injections and four 
compositions corresponding to the four
edges. 
The edges $(\emptyset, \{ 1\} )$ and $(\{ 1'\}, \{ 1',1 \})$ are {\em of the first kind}, and the edges
$(\emptyset, \{ 1'\})$ and $(\{ 1\}, \{ 1,1'\})$ are {\em of the second kind}.

\begin{definition}\label{def:G1}
For a linear set $U \subset V$, we define sets, zero section, projections
\begin{align*}
\sfG^{\sfone;\sfone}U  & = U^\sett{1} = \{ (v_0,v_1,t_1) \in V^2 \times \K \mid \, v_0 \in U, v_0 + v_1 t_1 \in U \} ,
\cr
\sfG^{\emptyset;\sfone} U  & = U \times \K
\cr
z^{\emptyset,\sfone;\sfone}, :  \sfG^{\emptyset;\sfone} U  \to  \sfG^{\sfone;\sfone} U, & \quad
(v_0,t_1) \mapsto (v_0,0,t_1)
\cr
\pi^{\emptyset,\sfone;\sfone}_0 : \sfG^{\sfone;\sfone}U \to \sfG^{\emptyset;\sfone} U
& \quad
 (v_0,v_1,t_1)  \mapsto (v_0 ,t_1), 
\cr
\pi^{\emptyset,\sfone;\sfone}_1 : \sfG^{\sfone;\sfone}U \to \sfG^{\emptyset;\sfone} U
& \quad (v_0,v_1,t_1) \mapsto  (v_0 + t_1 v_1,t_1),
\end{align*}
and composition  given on $U^\sfone \times_{U\times \K} U^\sfone$ by
 by
$$
(v_0',v_1',t_1) \ast (v_0,v_1,t)= (v_0,v_1+v_1',t).
$$
These data define a groupoid, denoted by $U^\sett{1}$ in Part I,
$$
\sfG^\sfone U := (\sfG^{\emptyset;\sfone}U,\sfG^{\sfone;\sfone}U,\pi^{\emptyset,\sfone;\sfone}_\sigma,z^{\emptyset,\sfone;\sfone}, *) .
$$
For fixed $t_1=t$, we use the notation $\sfG_t^\sfone U$ for the groupoid obtained by freezing the last variable. Its morphism set is
$$
\sfG_t^{\sfone;\sfone} U = U_t = \{ (v_0,v_1) \in V^2 \mid \, (v_0,v_1,t) \in U^\sfone \}.
$$
The description by triples $(v_0,v_1,t_1)$ will be called the {\em affine coordinates} of $\sfG^\sfone U$.
\end{definition}

\nin
We shall give explicit formulae with respect to affine coordinates also at higher order; however, note that affine coordinates are
not an intrinsic feature (in geometric terms, they depend on the flat affine connection given by a chart domain).

\begin{definition}\label{def:G2}
For a linear set $U \subset V$, we define vertex sets 
\begin{align*}
\sfG^{\ol \sfone;\ol \sfone}U  & = U^\settt{1} = \{ (v_0,v_1,s_1,t_1) \in V^2 \times \K^2 \mid \, v_0 \in U, v_0 + v_1 s_1 t_1 \in U \} ,
\cr
\sfG^{\emptyset;\ol \sfone} U  & = U \times \K ,
\cr
\sfG^{\sfone;\ol \sfone}U  & = \sfG^{\sfone;\sfone}U  =U^\sett{1}  ,
\cr
\sfG^{\{ 1' \} ;\ol \sfone} U  & = U \times \K^2 ,
\end{align*}
and edge projections $\pi^{\beta,\alpha;\ol \sfone}_\sigma : \sfG^{\alpha;\ol \sfone} U\to \sfG^{\beta; \ol \sfone}U$
for each of the four edges $(\beta,\alpha)$, as described in Part I, Section 3.2, as well as zero sections and four compositions, as defined in loc.\ cit.
These data define
a small double category $\sfG^{\ol \sfone} U$ (and denoted by $U^\settt{1}$ from Part I).
Sometimes, as in Part I,
projections and composition corresponding to the two edges of second kind will be denoted by
$\partial$ and $\bullet$. 
\end{definition}

\begin{definition}\label{def:point1}
By {\em terminal object} we mean the linear set $0 = (0,0)$ (zero $\K$-module). 
Recall from Part I that $\sfG^\sfone 0 = (\K,\K)$ has a trivial groupoid structure, and that  the double cat
$\sfG^{\ol \sfone} 0$ reduces to the left action category $(\K,\K^2)$ of $\K$; recall also that there are canonical
morphisms of groupoids
$$
0_U: \sfG^\sfone U \to \sfG^\sfone 0, \qquad \ol 0_U: \sfG^{\ol \sfone} U \to \sfG^{\ol \sfone} 0 .
$$
\end{definition}

\begin{definition}\label{def:finitepart1}
The {\em finite part} of $\sfG^\sfone U$, denoted by $\SfG^\sfone U$, is the subgroupoid lying over the invertible scalars
$\K^\times$, and denoted by $(U^\sett{1})^\times$ in Part I.
The {\em finite part} of $\sfG^{\ol \sfone} U$, denoted by $\SfG^{\ol\sfone} U$, is the sub-doublecat lying over 
$\K^\times \times \K^\times \downdownarrows \K^\times$, and denoted by $(\bfU^\settt{1})^\times$ in 
Theorem 3.6 of Part I. Recall from that theorem that the finite part is isomorphic to the direct product of the pair groupoid of $U$
with the one of $\K^\times$. 
\end{definition}

\begin{theorem}\label{th:top1}
Under the assumptions of topological differential calculus, for  a map $f:U \to W$ the following are equivalent:
\begin{enumerate}
\item
$f$ is of class $C^1_\K$,
\item
$f$ induces a continuous morphism of groupoids 
$\sfG^\sfone f : \sfG^\sfone U \to \sfG^\sfone W$
such that $0_W \circ \sfG^\sfone f = 0_U$, 
\item
$f$ induces a continuous morphism of small double cats
$\sfG^{\ol \sfone} f : \sfG^{\ol\sfone} U \to \sfG^{\ol \sfone} W$
such that $\ol 0_W \circ \sfG^{\ol\sfone} f = \ol 0_U$.
 \end{enumerate}
\end{theorem}

\begin{proof}
The two {\em vertex maps} of $\sfG^\sfone f$ will be denoted by
$$
\sfG^{\sfone;\sfone} f:\sfG^{\sfone;\sfone} U \to \sfG^{\sfone;\sfone}  W , \qquad 
\sfG^{\emptyset;\sfone} f:  \sfG^{\emptyset;\sfone} U\to \sfG^{\emptyset;\sfone}W ,
$$
and likewise for the four vertex maps of $\sfG^{\ol \sfone} f$.
As seen in Part I, in affine coordinates, $\sfG^{\sfone;\sfone} f$ is neccesarily given by
$$
f^\sett{1} (v_0,v_1,t_1) = (f(v_0), f^{[1]}(v_0,v_1,t_1), t_1),
$$
where $f^{[1]}$ is a {\em difference factorizer} of $f$.
If $f$ is $C^1$, 
 a {\em continuous} difference factorizer exists, and it is automatically
additive and homogeneous, hence $f$ induces morphisms of groupoids and of double cats. 
Conversely, if a continuous $\sfG^\sfone f$ exists, then a continuous $f^{[1]}$ exists, and $f$ is $C^1$.
\end{proof}

\begin{definition}\label{def:ccatC1}
Let $\K$ be an arbitary ring and $(U,V)$, $(U',V')$ be linear sets. A {\em law of class $C^1$} is a morphism
$\bff^\sfone = (\bff^{\emptyset;1},\bff^{\sfone;\sfone})$ from $\sfG^\sfone U$ to $\sfG^\sfone U'$ such that
$0_{U'} \circ \bff^\sfone = 0_U$, and a
{\em homogeneous} law is a morphism
$\bff^{\ol\sfone}$ from $\sfG^{\ol \sfone} U$ to $\sfG^{\ol \sfone} U'$
such that $\ol 0_{U'} \circ \bff^{\ol\sfone}  = \ol 0_U$.
The {\em concrete category of linear sets with $C^1$-laws} is given by linear sets
$((U,V);( \sfG^{\emptyset;\sfone} U, \sfG^{\emptyset;\sfone} V);( \sfG^{\sfone ; \sfone} U,\sfG^{\sfone;\sfone} V))$ and morphisms
$(f, \bff^{\emptyset;1},\bff^{\sfone;\sfone})$. In this ccat, {\em pullbacks} are defined by Definition \ref{def:pullback}.
\end{definition}

\begin{theorem}\label{th:polyone}
Every polynomial law induces a homogeneous law of class $C^1$.
\end{theorem} 

\begin{proof} Part I, Theorem 4.15.
\end{proof}

\subsection{Copies of $k$-th generation}\label{sscec:k-generation} We define
the symbols  $\sfG^\sett{k}$ and $\K_k$ to be formal copies of $\sfG^\sett{1}$ and $\K$, 
together with a label $k\in \N$ remembering its ``generation'':  
\begin{align}
\sfG^\sett{k} U & :=  (\sfG^{\emptyset;\{ k \}} U, \sfG^{\{ k \};\{ k\}} U, \pi_\sigma^{\emptyset,\{ k \};\{ k\}}, z, *)
\cr
& :=  (U \times \K_\K, U^\sett{k}, \pi_\sigma,z,\ast)  
 \end{align}
is a formal copy, with label $k$, of the groupoid $\sfG^\sfone U = (U\times \K,U^\sett{1})$ (def.\ \ref{def:G1}). We say that this
copy is {\em of $k$-th generation}. In the same way, for $f:U \to U'$, 
 the {\em copy of $k$-th generation of $f^\sett{1}$} is 
\begin{equation}
\sfG^\sett{k} f := (f \times \id_\K,  f^\sett{k} ) : \sfG^\sett{k} U \to \sfG^\sett{k} U' ,
\end{equation}
Similarly,  $\sfG^\settt{k}$ and $f^\settt{k}$ are  defined as a formal copy of $\sfG^{\ol \sfone}$ and $f^\settt{1}$.

\subsection{The $n$-fold groupoid $\sfG^n U$}
To define the symbol $\sfG^\sfn=\sfG^\sett{1,\ldots,n}$, 
the idea is simply to apply first $\sfG^\sett{1}$, then $\sfG^\sett{2}$, up to $\sfG^\sett{n}$,
to the initial datum $(U,V)$. At the first step this gives a pair of linear sets, forming a groupoid,
 at the second step four  linear sets, forming a double groupoid, and at the $n$-th step 
a hypercube of $2^n$ linear sets,  which form an $n$-fold groupoid (cf.\ Appendix \ref{app:nfoldcats} for definitions). 

\begin{theorem}[The full $n$-fold groupoid]\label{th:Full1}
Let $(U,V)$ be a linear set. Then, for all  $n \in \N$, there exist  uniquely  defined $n$-fold groupoids
$$
\sfG^\sfn U = \bigl( \sfG^{\alpha; \sfn} U \bigr)_{\alpha \in \cP(\sfn)} ,
$$
such that $\sfG^\sfone U$ is as in Definition \ref{def:G1}, and 
$\sfG^\sfn U$ is the $n$-fold groupoid 
obtained by applying $\sfG^\sett{n}$ to the data of the $(n-1)$-fold groupoid $\sfG^{\mathsf{n-1}} U$, that is,
$$
\sfG^\sfn U = \sfG^\sett{1,\ldots,n}U := \sfG^\sett{n} (\sfG^\sett{1,\ldots,n-1} U) = (\sfG^\sett{n} \circ \sfG^\sett{n-1} \circ \cdots \circ \sfG^\sett{1}) U .
$$
In particular, projections and zero section in the following diagram are morphisms of $n-1$-fold groupoids:
\begin{equation} \label{eqn:n-1}
\begin{matrix}
\sfG^\sfn U = \sfG^\sett{n} (\sfG^{\mathsf{n-1}}U ) \qquad { }  \cr
\pi^n_\sigma \downdownarrows \quad \uparrow z^n \cr
\sfG^{\mathsf{n-1}}U \times \K_n 
\end{matrix}
\end{equation}
Analogous statements hold for the {\em finite part} $n$-fold groupoid
$\SfG^\sfn U$, which is defined inductively by
$
\SfG^\sfn U = \SfG^\sett{n} (\SfG^\sett{1,\ldots,n-1} U)= (\SfG^\sett{n} \circ \SfG^\sett{n-1} \circ \cdots \circ \SfG^\sett{1}) U$.
As $n$-fold groupoid, it is a product of iterated pair groupois (Appendix \ref{app:pg}),
$$
\SfG^\sfn U \cong
\pG^\sfn U \times \pG^\sett{2,3,\ldots,n} \K_1^\times \times \pG^\sett{3,4,\ldots,n} \K_2^\times  \times \ldots \times  
\pG^\sett{n} \K_{n-1}^\times \times \K_n^\times .
$$
\end{theorem}

\begin{proof}
Projections, injections and product $\ast_1$ of the groupoid $(\sfG^\sfone U,\ast_1)$ are given by polynomial (quadratic) formulae.
Hence we can apply $\sfG^\sett{2}$ to these data, and we get again data of the same kind, i.e., a groupoid with product
$\sfG^\sett{2}(\ast_1)$. Note that the product
$\ast_1^\sett{2}$ is defined on the set $(U^\sett{1}\times_{U\times \K} U^\sett{1})^\sett{2}$, which, according to Theorem
\ref{la:pullback}, contains the domain of definition
$$
D:= \sfG^\sett{2} (U^\sett{1}\times_{U\times \K} U^\sett{1}) =
\sfG^\sett{2} U^\sett{1}\times_{\sfG^\sett{2} (U\times \K)} \sfG^\sett{2}  U^\sett{1}
$$
of $\sfG^\sett{2}(\ast_1)$, and hence is well-defined. Once the structure maps are seen to be well-defined, 
it is immediate from the functorial properties of $\sfG^\sett{2}$ that
$\sfG^\sett{2} (\sfG^\sfone,\ast_1)$ is again a groupoid. 
To show that $\sfG^\sett{2} \sfG^\sfone U$ is a double groupoid, note that it carries also a groupoid structure with
product $\ast_2$ over the base $\sfG^\sfone U \times \K$, and this groupoid structure is compatible with the preceding one
since its structure maps are all of the form $\sfG^\sett{2}f$ for some $f$, and such maps are compatible with the
$\ast_2$-groupoid structure. Theorem \ref{la:pullback} is used to ensure that $D$ is indeed stable under
$\ast_2$. 

\ssk
The claim for general $n\in \N$ is now proved by induction, using exactly the same arguments: 
applying $\sfG^\sett{n}$ to the $(n-1)$-fold groupoid structure of $\sfG^{\mathsf{n-1}} U$ yields again an
$(n-1)$-fold groupoid; concerning domains of definitions of partially defined products, this is well defined since 
the functor $\sfG^\sett{n}$ acts on pullbacks as defined in Definition \ref{def:pullback}.
The $(n-1)$-fold groupoid is  of type $\ul{\rm Goid}$ since all structures are compatible with the groupoid structure
$\ast_n$ coming with the functor $\sfG^\sett{n}$.  Taken together, this shows that $\sfG^\sfn U$ is an $n$-fold groupoid.

\ssk
For $\SfG$, we proceed in the same way by induction, starting with Def.\   \ref{def:finitepart1}.
As shown in Part I, $\SfG^\sfone U \cong \pG^\sfone U \times \K_1^\times$, so
$$
\SfG^\sftwo U \cong \pG^\sett{2} (\pG^\sett{1} U \times \K_1^\times) \times \K_2^\times 
= \pG^\sett{1,2} U \times \pG^\sett{2} \K_1^\times  \times \K_2^\times,
$$
and so on, by induction.
\end{proof}


\nin
The $2^n$ vertex sets $\sfG^{\alpha;\sfn} U$ of the $n$-fold groupoid will be described later in more detail
(Section \ref{sec:full}, Theorem  \ref{th:fullcompute2}). 
For the moment, let us just record the description of the {\em top vertex set} ($\alpha = \sfn$) and of the
{\em bottom vertex set} ($\alpha = \emptyset$):

\begin{theorem}[Top and bottom vertex sets]\label{th:topvertex}
The top vertex set of $\sfG^\sfn U$ is given by
$$
\sfG^{\sfn;\sfn} U =  U^\sett{1,\ldots,n} :
 = (U^\sett{1,\ldots,n-1})^\sett{n} .
$$ 
Up to notation, this is the set $U^{[n]}$ from Definition \ref{def:full}. 
The bottom vertex set  is
$$
\sfG^{\emptyset;\sfn} U = U \times \K^n .
$$
The finite part $n$-fold groupoid is the sub-$n$-fold groupoid of $\sfG^\sfn U$ projecting onto its bottom vertex set 
$$
\SfG^{\emptyset;\sfn} U = U \times (\K^\times)^n .
$$
\end{theorem}

\begin{proof}
The top vertex set of $\sfG^\sfn U$ is obtained by applying $\sfG^\sett{n}$ to the top vertex set of $\sfG^{\mathsf n-1} U$.
Since the top vertex set of $\sfG^\sfone U$ is $U^\sett{1}$, the claim follows by induction.
The bottom vertex set of $\sfG^\sfn U$ is a direct product of $\K$ with the bottom vertex set of 
$\sfG^{\mathsf{n-1}} U$. Since the bottom vertex set of $\sfG^\sfone U$ is $U \times \K$ and the one of
$\SfG^\sfone U$ is $U \times \K^\times$, the claims about the bottom vertex sets follow again by induction.
\end{proof}

\nin 
For instance, the top vertex set of $\sfG^\sftwo U$ is given in affine coordinates by
\begin{equation}\label{eqn:U-two}
U^\sftwo = (U^\sett{1})^\sett{2} 
= 
\end{equation}
\begin{equation*}
\Bigsetof{ (v_0,v_1,v_2,v_{12},t_1,t_2,t_{12}) \in V^4 \times \K^3 } {\begin{array}{c}
v_0 \in U \\  v_0 + t_1 v_1 \in U,  \\
  v_0 + t_2 v_2 \in U \\ v_0 + t_2 v_2 +  (t_1+t_2 t_{12} )(v_1+  t_2 v_{12}) \in U \end{array}}
\end{equation*} 
(compare this with (\ref{eqn:02})!). 
Note that is formula is ``not symmetric in the indices $1,2$'', indicating that the $n$-fold groupoid
$\sfG^\sfn U$ is  {\em not edge-symmeric}. Therefore,
it is important to have  notation taking account of issues related to the order of $\N$:

\begin{definition}\label{def:N-def}
Let $N = \{ a_1, \ldots ,a_n \} \subset \N$ be a non-empty finite subset. We consider $N$ as ordered, with the natural order inherited
 from $\N$:  we assume that $a_1 < \ldots < a_n$. 
The $n$-fold groupoid $\sfG^N U$ is defined as in the preceding theorem by
$$
\sfG^N U := (\sfG^{\{ a_n \}} \circ \cdots \circ \sfG^{\{ a_1 \} }) U .
$$
For each vertex $\alpha$, the {\em vertex spaces} are denoted by $\sfG^{\alpha;N} U$;
the {\em top vertex set} is 
$$
\sfG^{N;N}U = U^N =  ( \ldots (U^\sett{a_1})^\sett{a_2} \ldots )^\sett{a_n} 
$$ 
and the {\em bottom vertex set} is $\sfG^{\emptyset;N}U = U \times \K^n$.
For each edge
$(\beta,\alpha)$, the {\em edge projections and injections} are denoted by
$\pi^{\beta,\alpha;N}_U$, resp.\
$z^{\beta,\alpha;N}_U$, and the groupoid law belonging to the edge category is denoted by
$\ast^{\beta,\alpha;N}$. More generally, if $(\gamma,\alpha)$ is a $k$-cube in $\cP(N)$ (so $\gamma \subset \alpha$
and $\vert \alpha \setminus \gamma \vert =k$),
there is a unique section 
$$
z^{\gamma,\alpha;N}: \sfG^{\gamma;N} U \to \sfG^{\alpha;N} U ,
$$
and there are $2^k$ projections $\pi^{\gamma,\alpha;N}_{\bf \sigma}: \sfG^{\alpha;N} U \to \sfG^{\gamma;N} U$
(at each step there is a choice between target and source projection).
The section $z^{\emptyset,N;N}$ is called the {\em bottom-up section}, and the $2^n$ projections
$\pi^{\emptyset,N;N}$ are called the {\em top-down projections}.
\end{definition}

\nin This definition makes it obvious that, if $N = A \cup B$ is a ``Dedekind cut'' (by this we mean:
$\forall a \in A, \forall b \in B$: $a<b$), then
\begin{equation}
\sfG^{A \cup B} = \sfG^B \circ \sfG^A , \qquad U^{A \cup B} = (U^A)^B . 
\end{equation}

\subsection{The $2n$-fold category $\sfG^{\ol n} U$}
To define the symbol $\sfG^{\ol \sfn}=\sfG^\sett{1,1',\ldots,n,n'}$, 
we proceed as above, applying at each step the double-cat rule $\sfG^{\ol{ \{ k \} }}$:

\begin{theorem}\label{th:Full2}
Let $(U,V)$ be a linear set and $n \in \N$. Then there exists a unique $2n$-fold category, given by data as in Theorem
\ref{th:nfoldcat} of the form
$$
\sfG^{\ol \sfn} U = \bigl( \sfG^{\alpha; \ol \sfn} U \bigr)_{\alpha \in \cP(\ol\sfn)} ,
$$
such that $\sfG^{\ol \sfone} U = \sfG^\settt{1} U$ is the double category defined above, and 
$$
\sfG^{\ol \sfn} U = \sfG^\settt{1,\ldots,n}U := \sfG^\settt{n} (\sfG^\settt{1,\ldots,n-1} U) = (\sfG^\settt{n} \circ \sfG^\settt{n-1} \circ \cdots \circ \sfG^\settt{1}) U .
$$
\end{theorem}

\begin{proof}
The theorem is proved by induction, in the same way as Theorem \ref{th:Full1}, by applying at each step a functor of the type $\sfG^\settt{k}$ to
the structure from the preceding level (cf.\ Section \ref{sec:catconstruction}). 
\end{proof}

\nin
The symbols $\SfG^N$ and $\sfG^{\ol N}$ can now be defined similarly as in  \ref{def:N-def}, and the structure of the finite part
can be described similarly as in Theorem \ref{th:Full1}.

\subsection{Morphisms, and laws}

\begin{definition}\label{def:law}
Assume $(U,V)$, $(U',V')$ are linear sets over $\K$. 
A {\em $C^n$-law between $U$ and $U'$} is given by  morphisms of $k$-fold groupoids
$\bff^\sfk : \sfG^\sfk U \to \sfG^\sfk U'$ for $k=0,\ldots,n$, compatible with each other, in the sense that
they complete diagram (\ref{eqn:n-1})
\begin{equation} \label{eqn:fn-1}
\begin{matrix}
\sfG^\sfk U = \sfG^\sett{k} (\sfG^{\mathsf{k-1}}U ) \qquad { } & \buildrel{\bff^\sfk} \over \to & \sfG^\sfk U'  \cr
\pi^n_\sigma \downdownarrows \quad \uparrow z^n  & & \downdownarrows \quad \uparrow \cr 
\sfG^{\mathsf{k-1}}U \times \K_k & \buildrel{\bff^{\mathsf{k-1}} \times \id_\K} \over \to & \sfG^{\mathsf{k-1}}U' \times \K_k.
\end{matrix}
\end{equation}
The set map 
$f:= \bff^\emptyset : U\to U'$ is called the
 {\em base map of $\bff^\sfn$}.
A {\em $C^\infty$-law between $U$ and $U'$} is given by 
morphisms of $k$-fold groupoids
$\bff^\sfk : \sfG^\sfk U \to \sfG^\sfk U'$
 for each $k \in \N$ which are compatible with each other, in the sense
explained above.
In the same way, {\em homogeneous laws} are defined by replacing $\sfn$ by $\ol \sfn$ and morphisms of $n$-fold
groupoids by morphisms of $2n$-fold categories. 
\end{definition}

\begin{definition}\label{def:Cn-cat}
We denote by
\ul{$C^n_\K$-{\rm linset}} (resp. \ul{$C^{\ol n}_\K$-{\rm linset}})
the concrete category of linear sets with (homogeneous) $C^n$-laws as morphisms:
objects of \ul{$C^n_\K$-{\rm linset}} are families of linear sets
$(U,\sfG^\sfn U; V , \sfG^\sfn V)$ and morphisms are $C^n$-laws $(f,\eff^\sfn)$.
Pullbacks in this category are defined by
the base linear set $A \times_C B$ (see Definition \ref{def:pullback}) and the extended sets 
$$
\sfG^\sfn (A \times_C B ) := \sfG^\sfn A \times_{\sfG^\sfn C} \sfG^\sfn B
$$
(which form again an $n$-fold groupoid, by induction from Theorem \ref{la:pullback}). 
In particular, for $C = 0$ we get the behaviour of $\sfG^\sfn$ with respect to cartesian products
$$
\sfG^\sfn (A \times B ) = \sfG^\sfn A \times_{\sfG^\sfn 0} \sfG^\sfn B .
$$
Similar definitions are given for the ccat of linear sets with
 homogeneous $C^n$-laws, with the symbol $\sfG^\sfn$ replaced by $\sfG^{\ol \sfn}$.
\end{definition}

\begin{theorem}\label{th:polylaws}
Assume $V$ and $W$ are $\K$-modules and 
$P:\ul V \to \ul W$ is a polynomial law  (cf.\ Part I, Def.\ 4.12).
Then $P$ induces homogeneous $C^\infty$-laws. 
\end{theorem}

\begin{proof}
By induction from Theorem \ref{th:polyone}.
\end{proof}

\nin
For general laws, the base map does not always determine uniquely the whole law; however, on the level of the finite part,
they do:

\begin{theorem}\label{th:finitelaws}
Every set-map $f:U \to U'$ induces unique morphisms of finite part
$n$-fold groupoids, resp.\ $2n$-fold cats, compatible with each other, such  that $f=\bff^\emptyset$.
\end{theorem}

\begin{proof}
For $n=1$, recall from Part I, that, over $\K^\times$,
 the difference factorizer can be expressed as a difference quotient of the base map, and is uniquely defined in this way.
 For $n>1$, the claim now follows by induction. 
\end{proof}

\nin
In the topological setting, we can replace $\SfG$ by $\sfG$, ``by density'':

\begin{theorem}\label{th:topCn}
Assume $\K$ is a topological ring with dense unit group and $(U,V)$, $(U',V')$ are non-empty open linear sets in topological
$\K$-modules.
Then, for all $n \in \N\cup \{\infty\}$,  we have a bijection between

\begin{itemize}
\item
 $C^n$-maps $f:U \to U'$,
\item 
continuous $C^n$-laws $\bff^\sfn : \sfG^\sfn U \to \sfG^\sfn U'$, 
\item
continuous homogeneous $C^n$-laws $\bff^{\ol\sfn} : \sfG^{\ol\sfn} U \to \sfG^{\ol\sfn} U'$.
\end{itemize}
The $C^n$-laws $\bff^\sfn$, resp.\ $\bff^{\ol \sfn}$ are then entirely defined by the base map $f$.
\end{theorem}

\begin{proof}
For $n=1$, this is Theorem \ref{th:top1}, and
for $n>1$, it follows by induction.
\end{proof}

\begin{definition}\label{def:derived} For $f:U \to U'$ as in Theorem \ref{th:finitelaws}, the unique laws determined by $f$
will be denoted
$$
\SfG^\sfn f : \SfG^\sfn U \to \SfG^\sfn U' , \qquad \SfG^{\ol \sfn} f : \SfG^{\ol\sfn} U \to \SfG^{\ol\sfn} U' ,
$$
and if $f$ is a $C^n$-map as  in Theorem \ref{th:topCn}, the unique  induced laws are denoted by
$$
\sfG^\sfn f := \bff^\sfn: \sfG^\sfn U \to \sfG^\sfn U', \qquad
\sfG^{\ol \sfn} f:= \bff^{\ol \sfn} :\sfG^{\ol\sfn} U \to \sfG^{\ol\sfn} U' .
$$
In both situations,   we say that these laws are obtained ``by deriving the map $f$''.
We also use such notation and language when $f$ is polynomial law, and especially
when $f$ is constant, linear or bilinear. 
\end{definition}

\nin
The ``derived laws'' of constant, linear and bilinear polynomials are computed by induction, using the results for
$n=1$ in Part I, Section 4.

\subsection{Scaleoids, and the terminal map} 

\begin{definition}
By {\em scaledoid}, we mean the family of $n$-fold groupoids, resp.\ of $2n$-fold categories 
$
\sfG^\sfn 0,  \sfG^{\ol \sfn} 0$,
for $n\in \N$,
where $0=(\{ 0 \},\{ 0\})$ is the zero-space, seen as linear set over $\K$. 
\end{definition}

\nin
For $n=1$, we have $\sfG^\sfone 0 = (\K, \K)$ with trivial groupoid structure, so, by induction, 
$\sfG^\sfn 0$ really is an $(n-1)$-fold groupoid.
Likewise, $\sfG^{\ol \sfone} 0 = (\K , \K^2)$ is the left action category of $\K$ on itself (Part I, Theorem 3.6), and so
$\sfG^{\ol \sfn} 0$ really is a small $(2n-1)$-fold cat. 
Its structure is complicated; however, for the finite part, we easily get

\begin{theorem}
The finite part of $\sfG^\sfn 0$ is isomorphic
to a product of iterated pair groupoids
$$
\SfG^\sfn 0 \cong
 \pG^\sett{2,3,\ldots,n} \K_1^\times \times \pG^\sett{3,4,\ldots,n} \K_2^\times  \times \ldots \times  
\pG^\sett{n} \K_{n-1}^\times \times \K_n^\times .
$$
\end{theorem}

\begin{proof}
This follows from Theorem \ref{th:Full1} for $U = 0$.
\end{proof}

\begin{theorem}[The terminal map] 
Assume $(U,V)$ is a linear set.
\begin{enumerate}
\item
The constant map $0_U:U \to 0$ induces [homogeneous] $C^\infty$ laws
$$
0_U^\sfn: \sfG^\sfn U \to \sfG^\sfn 0 , \qquad 
0_U^{\ol \sfn}: \sfG^{\ol \sfn} U \to \sfG^{\ol\sfn} 0 . 
$$
\item
For each $x \in U$, the map $0 \to U$, $0 \mapsto x$ induces [homogeneous] $C^\infty$-laws
$\sfG^\sfn 0 \to \sfG^\sfn U$ and
$\sfG^{\ol \sfn} 0 \to \sfG^{\ol \sfn}0$ which are sections of the morphisms from (1).
\item
Every $C^n$-law $\bff^\sfn$ is compatible with terminal maps  in the sense  that 
$$
0_{U'}^\sfn \circ \bff^\sfn = 0_U^\sfn.
$$
\end{enumerate}
\end{theorem}

\begin{proof}
(1) and (2) follow from Theorem \ref{th:polylaws} and  the fact that constant maps are polynomial, and (3) follows from the compatibility 
condition (\ref{eqn:fn-1}): for $n=1$, it implies that $f^\sett{1}(x,v,t)=(f(x),f^{[1]}(x,v,t),t)$, whence
$0_{U'}^\sfone \circ \bff^\sfone = 0_U^\sfone$, and by induction the claim follows. 
\end{proof}

\subsection{$C^\infty$-manifold laws, and general $C^\infty$-spaces}\label{sec:manifolds}
Following Subsection 5.1 of Part I, we may now define {\em $C^n$-manifold laws over $\K$}, and prove that,
if the manifold law is handy, it gives rise
to an $n$-fold groupoid $\sfG^\sfn M$, which we call the {\em $n$-fold magnification of $M$},
 resp.\ a small $2n$-fold cat $\sfG^{\ol \sfn} M$,
 everything projecting to the scaleoid
$\sfG^\sfn 0$, resp.\ to $\sfG^{\ol \sfn}0$. 
Moreover, the {\em finite part} $\SfG^\sfn M$ admits a decomposition into interated pair groupoids, as stated for $M=U$ in
Theorem \ref{th:Full1}.
If the manifold law is not assumed to be handy, all statements remain valid for {\em local} structures, as in Part I. 

\ssk
Since the general category of $C^\infty$-manifolds thus defined is not cartesian closed and lacks existence of inverse images
(Part I, Section 5.2),
one will wish to define more general categories of $C^\infty$-spaces:
these shall be concrete categories of small $2n$-fold cats, projecting to the final object ``scaleoid'',
 and containing finite parts of iterated pair groupoids lying over the finite part of the scaleoid.
I will come back to such issues in subsequent work.


\section{Symmetric cubic calculus}\label{sec:sym}

\subsection{The $n$-fold groupoids $\sfGsy^\sfn_\ttt U$ and  $\sfGsy^\sfn U$}
The basic principle of symmetric cubic
 calculus is to apply the cat rule  $\sfG_t$ (Definition \ref{def:G1}) over and over again, by varying the value of
$t$, so that after $n$ steps we get an $n$-fold cat rule  $\sfGsy^\sfn_\ttt := \sfG_{t_n}\circ \ldots \circ \sfG_{t_1}$. 
On the technical level, the  procedure is simpler than in full
cubic calculus since each of the functors $\sfG_t$ {\em preserves cartesian products} (they are examples of 
``product preserving functors'' in the sense of \cite{KMS93}), so the discussion of pullbacks
from Appendix \ref{app:pullback} can be skipped. As a drawback, implementation of ``homogeneity'' into the set-up is less natural than in the
preceding section, so we will mainly speak about $n$-fold groupoids and not  about $2n$-fold cats.

\begin{theorem}[The symmetric cubic $n$-fold groupoids]\label{th:Sym1}
Let $(U,V)$ be a linear set and $n \in \N$, and fix $\ttt = (t_1,\ldots,t_n)\in \K^n$.
Then there exists a unique $n$-fold groupoid
$$
\sfGsy^\sfn_\ttt U = \bigl( \sfGsy_\ttt^{\alpha; \sfn} U \bigr)_{\alpha \in \cP(\sfn)} ,
$$
such that $\sfGsy^\sfone_t U$ is the groupoid defined in Definition \ref{def:G1}, and 
$\sfGsy_\ttt^\sfn U$ is the $n$-fold groupoid 
obtained by applying $\sfG^\sett{n}_{t_n}$ to the data of the $(n-1)$-fold 
groupoid $\sfGsy^{\mathsf{n-1}}_{t_1,\ldots,t_{n-1}}U$, that is,
$$
\sfGsy^\sfn_\ttt U = \sfGsy^\sett{1,\ldots,n}_\ttt U := \sfG^\sett{n}_{t_n} (\sfGsy^\sett{1,\ldots,n-1}_{t_1,\ldots,t_{n-1}}U) 
= (\sfG^\sett{n}_{t_n} \circ \sfG^\sett{n-1}_{t_{n-1}} \circ \cdots \circ \sfG^\sett{1}_{t_n}) U .
$$
\end{theorem}

\begin{proof}
We use the same arguments as in the proof of Theorem \ref{th:Full1}. 
\end{proof}

Next, to  describe vertex sets and edge projections we need the following definition.
A by-product of the computations will be the non-trivial fact that there is a natural action of the group $S_n$
of permutations on $n$ letters.

\begin{definition}[The ring action]\label{def:K-action}
For any $\alpha = \{ \alpha_1,\ldots,\alpha_\ell \} \in \cP(\sfn)$, let $V_\alpha$ be a copy of $V$, on which we let act the ring $\K^n$
(direct product of rings) in the following way:
$$
\forall \ttt \in \K^n, \forall v_\alpha \in V_\alpha : \quad
\ttt . v_\alpha :=  t_{\alpha_1} \cdots t_{\alpha_\ell} \cdot   v_\alpha .
$$
\end{definition}

\begin{theorem}[Structure of symmetric $n$-fold groupoids]\label{th:Sym2}
Let $(U,V)$ be a linear set and $n \in \N$, and fix $\ttt = (t_1,\ldots,t_n)\in \K^n$.
\begin{enumerate}
\item
For $\alpha \in \cP(\sfn)$, the vertex set $\sfGsy_\ttt^{\alpha;\sfn} U$ of $\sfGsy^\sfn_\ttt U$
is given by
$$
U_\ttt^{\alpha;\sfn}:=
\Bigl\{
(v_\beta)_{\beta \in \cP(\alpha)}  \mid \, \forall \beta \in \cP(\alpha):
v_\beta \in V_\beta, \, 
\sum_{ \gamma  \in \cP(\beta)} \ttt.  v_\gamma
\in U \Bigr\} .
$$
The bottom vertex set is $U_\ttt^{\emptyset;\sfn} = U$.
\item
For an edge $(\beta,\alpha)$ with $\alpha = \beta \cup \{ i \}$, the two edge projections are
\begin{align*}
\pi_0 \bigl( (v_\gamma)_{\gamma \in \cP(\alpha)} \bigr) &= (v_\gamma)_{\gamma \in \cP(\beta)}
\cr
\pi_1 \bigl( (v_\gamma)_{\gamma \in \cP(\alpha)} \bigr) &= (v_\gamma + t_i v_{\gamma \cup \{ i\} })_{\gamma \in \cP(\beta)},
\end{align*}
the zero section $z=z_{\beta,\alpha;\sfn}$ comes from the natural inclusion $\cP(\beta) \subset \cP(\alpha)$:
$$
z  \bigl( (v_\gamma)_{\gamma \in \cP(\beta)} \bigr) = (v_\gamma)_{\gamma \in \cP(\beta)} ,
$$
and the corresponding groupoid composition law $\ast = \ast_{\beta,\alpha;\sfn}$ is given by
$$
(v_\gamma' )_{\gamma \in \cP(\alpha)} \ast  (v_\gamma)_{\gamma \in \cP(\alpha)} =
(v_\gamma, v_{\gamma\cup \{ i \} } + v_{\gamma \cup \{ i\} }' )_{\gamma \in \cP(\beta) }.
$$
\item
The $2^{\vert \alpha \vert}$ top-down projections from $U^{\alpha;\sfn}_\ttt$
 to the bottom vertex set obtained by composing target and source mappings are given by, for $\beta \in \cP(\alpha)$,
$$
\xi_{\alpha,\beta}: 
U^{\alpha,\sfn}_\ttt \to U, \quad 
(v_\gamma)_{\gamma \in \cP(\alpha)} \mapsto x_\beta := \sum_{\gamma 
\in \cP(\beta)} \ttt .  v_\gamma .
$$
\item
The ``flip'' or  ``exchange map'' is an isomorphism of double groupoids: 
$$
\sfGsy^\sftwo_{t_1,t_2} U \to \sfGsy^\sftwo_{t_2,t_1} U, \quad
(v_0,v_1,v_2,v_{12}) \mapsto (v_0,v_2,v_1,v_{12}).
$$
In particular, when $t_1=t_2$, the double groupoid is {\em edge-symmetric}. 
\item
For any $n \in \N$ and $\sigma \in S_n$, there is a natural  isomorphism 
$$
\sfGsy^\sfn_{(t_1,\ldots,t_n)} U \to \sfGsy^\sfn_{t_{\sigma(1)},\ldots,t_{\sigma(n)}} U, 
\quad 
(v_\alpha)_{\alpha \subset \sfn} \mapsto 
(v_{\sigma (\alpha)} )_{\alpha \subset \sfn} ,
$$
and for $t_1 = \ldots = t_n = t$, we get an edge-symmetric $n$-fold groupoid $\sfGsy^\sfn_\ttt U$.
\item
For two linear sets $U \subset V$, $U' \subset W$, and all $\ttt \in \K^n$, we have a natural isomorphism of $n$-fold groupoids
$$
\sfGsy^\sfn_\ttt (U \times U') \cong \sfGsy_\ttt U  \times \sfGsy_\ttt^\sfn U' .
$$
\end{enumerate}
\end{theorem}

\begin{proof}
All assertions are proved by induction.
Concerning (1), for $n=1$, this is simply another notation describing the groupoid $U_{t_1}$.
For $n=2$, let us give explicit formulae describing the double groupoid 
\begin{equation}\label{eqn:diagram}
\begin{matrix}
\sfGsy^{\sftwo;\sftwo}_{t_1,t_2} U  & \rightrightarrows & \sfGsy_{t_2}^{\{ 2 \}; \sftwo} U  \cr
 \downdownarrows & & \downdownarrows   
 \cr
\sfGsy^{\sfone;\sftwo}_{t_1} U  & \rightrightarrows &  U   .
 \end{matrix}
 \end{equation}
Applying Definition \ref{def:G1} in order to compute $U_{t_1,t_2}^{\sftwo,\sftwo} =(U_{t_1}^{[1]})_{t_2}^{[1]}$,  we get 
(cf.\ eqn.\  (\ref{eqn:U-two}))
\begin{equation}\label{eqn:Two}
U_{t_1,t_2}^{\sftwo,\sftwo} = \Bigsetof{ (v_0,v_1,v_2,v_{12}) \in V^4 } {\begin{array}{c}
v_0 \in U \\  v_0 + t_1 v_1 \in U,  \, \, v_0 + t_2 v_2 \in U \\ v_0 + t_1 v_1 + t_2 v_2 +  t_1 t_2 v_{12} \in U \end{array}}
\end{equation} 
projecting to $U_{t_1,t_2}^{\sfone,\sftwo} = \{ (v_0,v_1) \mid v_0 \in U, v_0 + t_1v_1 \in U \}$. 
On the other hand,
$U_{t_1,t_2}^{\{ 2\} ,\sftwo} = \{ (v_0,v_2) \mid v_0 \in U, v_0 + t_2v_2 \in U \}$ is obtained by extending $U$ 
with respect to $t_2$. This describes
the four vertex sets. We describe the four pairs of edge projections
$\pi^{\beta,\alpha;\sftwo}_\sigma$. For the edge $(\beta,\alpha)=(\sfone,\sftwo)=(\{ 1 \}, \{ 1,2 \})$, this is given directly
by the inductive definition:  for $\sigma = 1$ (target),
$$
\pi^{\{ 1 \}, \sftwo}_{t_1,t_2} (v_0,v_1,v_2,v_{12}) = (v_0,v_1) + t_2 (v_2,v_{12}) =
(v_0 + t_2 v_2, v_{2}+t_2 v_{12} ) 
$$
and for $\sigma = 0$ (source), $\pi^{\{ 1 \}, \sftwo}_{t_1,t_2} (v_0,v_1,v_2,v_{12}) = (v_0,v_1)$.
For the edge $(\beta,\alpha)=(\{ 2 \}, \{ 1,2 \})$, we have to apply the functor $\sfG_{t_2}$ to the edge
projection corresponding to the only non trivial edge of the $1$-cube:  for $\sigma = 1$ (target)
by ``deriving'' $v_0 + t_1 v_1$, we get  
$$
\pi^{ \{ 2 \}, \sftwo;\sftwo}_{t_1,t_2} (v_0,v_1,v_2,v_{12}) = (v_0 + t_1 v_1, v_2 +  t_1 v_{12} )
$$
(in full cubic calculus, there will  be here an additional term $t_{12} v_1 + t_2 t_{12} v_{12}$ to be added in the last component,
causing the more complicated structure of full cubic calculus). For $\sigma = 0$ (source), we get
$\pi^{ \{ 2 \}, \sftwo;\sftwo}_{t_1,t_2} (v_0,v_1,v_2,v_{12}) = (v_0,v_2)$. 
The other edge projections are, for $\sigma = 1$,
$$
\pi^{\emptyset,\sfone;\sftwo}_{t_1} (v_0,v_1)=v_0 + t_1 v_1, \qquad
\pi^{\emptyset,\{ 2 \};\sftwo}_{t_2}(v_0,v_2) = v_0 + t_2 v_2,
$$
and for $\sigma=0$ they just project to $v_0$. The groupoid laws are, in all cases, given by usual vector addition in certain
components of $\vvv$. 
Indeed, for the edges $(\emptyset,\sfone)$ and  $(\sfone,\sftwo)$ this follows again from the inductive definition
(and the definition of $\ast$ Part I), and for the remaining edges again we have to apply the functor $\sfG_{t_2}$ to the
groupoid law from the $1$-cube. But since vector addition  is a {\em linear} map $a: V \times V \to V$ (the commutativity of 
$(V,+)$ is important here),  we have $a^\sett{2} (x,u,t_2)=(a(x),a(u),t_2)$ (see Part I), so 
$\sfG_{t_2} a (x,u) = (a(x),a(u))$. 
Note that
 the four conditions appearing in (\ref{eqn:Two}) 
 correspond to the four projections of the top vertex set onto the bottom vertex set $U$, and that (4) is an immediate
 consequence of the explicit formulae. 

For $n \geq 2$, the inductive proof follows exactly the same arguments. 
\end{proof}

\nin
In order to state (5) in the form that $S_n$ acts on a certain $n$-fold groupoid, we incorporate
the $\ttt$-parameter into the groupoid, as we did for $n=1$ in Part I: 

\begin{definition}
For every vertex $\alpha \in \cP(\sfn)$ of the natural $n$-hypercube, 
we define 
\begin{equation}
\sfGsy^{\alpha;\sfn} U := \{ (\vvv,\ttt) \in \prod_{\alpha \in \cP(\sfn)} V_\alpha  \times \K^n  \mid \, \vvv \in \sfGsy^{\alpha;\sfn}_{\ttt} U  \} .
\end{equation}
Note that the {\em bottom vertex}, corresponding to $\alpha =\emptyset$, is
\begin{equation}
\sfGsy^{\emptyset;\sfn} U = U \times \K^n .
\end{equation}
Projections, zero sections and composition $\ast$ are defined as in the theorem, and with these structures
$\sfGsy^\sfn U$ is an edge-symmetric $n$-fold groupoid. It comes together with a family of  morphisms
\begin{equation}\label{eqn:sym-n}
\pi_n: \sfGsy^\sfn U \rightrightarrows \sfGsy^{\mathsf{n-1}} U \times \K 
\end{equation} 
and a morphism onto the trivial $n$-fold groupoid $\K^n$:
\begin{equation}
\sfGsy^\sfn U \to  \sfGsy^\sfn 0 = \K^n,\quad  (\vvv,\ttt) \mapsto \ttt
\end{equation}
The inverse image of $(\K^\times)^n$ under this morpism is called the {\em finite part} of $\sfGsy^\sfn U$:
\begin{equation}
\mathsf{Gsyfi}^{\alpha,\sfn} U := \{ (\vvv,\ttt) \in \sfGsy^{\alpha,\sfn} U \mid \forall i : t_i \in \K^\times \} .
\end{equation}
\end{definition}

\begin{theorem}
The finite part is a direct product of an $n$-fold pair groupoid with the trivial $n$-fold groupoid given by $(\K^\times)^n$:
$$
\mathsf{Gsyfi}^{\sfn} U \cong \pG^\sfn (U) \times (\K^\times)^n .
$$
On each vertex set, the isomorphism is given by the ``trivialization maps''
$$
\sfGsy^{\alpha;\sfn}_\ttt  U \to \pG^{\alpha;\sfn} U = \prod_{\gamma \in \cP(\alpha)} U, \quad
(v_\gamma)_{\gamma \in \cP(\alpha)}  \mapsto (x_\gamma)_{\gamma \in \cP(\alpha)},
$$
where $x_\gamma$ is defined as in Item (3) of Theorem \ref{th:Sym2}.
\end{theorem}

\begin{proof}
For $n=1$, this is Item (3) of Theorem 2.7, Part I, and for $n\in \N$, the claim follows by induction. 
In general, if an $n$-fold groupoid  is isomorphic to an $n$-fold pair groupoid, then the trivialization maps are given by
assembling the $2^n$ top-down projections. 
\end{proof} 

\subsection{Scalar action: homogeneity}
As explained in Part I, the groupoid structure corresponds to vector addition. Concerning  multiplication by scalars,
for various $\ttt$'s, the $n$-fold groupoids $\sfGsy_\ttt$ are related among each other by the following 

\begin{theorem}\label{th:Symscalar}
Let $\ttt, \sss = (s_1,\ldots,s_n) \in \K^n$. 
Then the family of maps, for $\alpha \in \cP(\sfn)$,
$$
\Phi_\sss^{\alpha} : \sfGsy_{\sss \cdot \ttt}^{\alpha,\sfn} U \to \sfGsy_{\ttt}^{\alpha,\sfn} U, \quad
(v_\gamma)_{\gamma \in \cP(\alpha)} \mapsto (\sss . v_\gamma)_{\gamma \in \cP(\alpha)}
$$
defines a morphism $\Phi_\sss$  of $n$-fold groupoids, where
$\sss \cdot \ttt = (s_1  t_1,\ldots , s_n t_n)$, and
$
\sss. v_\gamma = s_{\gamma_1}\cdots s_{\gamma_\ell} v_\gamma
$
is as in Definition \ref{def:K-action}. 
\end{theorem}

\begin{proof}
For $n=1$, this is Theorem 3.1 of Part I, and for general $n$, it follows by induction.
\end{proof}

\begin{remark}
Like in Chapter 3 of Part I, using the action of $\K^n$ from the preceding theorem, one can build an $(n+1)$-fold cat out of 
the $n$-fold groupoid $\sfGsy^\sfn U$. This $(n+1)$-fold cat is a subcat of the $2n$-fold cat $\sfG^{\ol \sfn} U$
from the preceding chapter; however, it seems to be less interesting than the whole $2n$-fold cat.
For $U=0$, the $(n+1)$-fold cat reduces to the left action category of $\K^n$ on itself.
\end{remark}

\begin{theorem} \label{th:tangentbundle}
The {\em $n$-fold tangent bundle of $U$},
$$
\sfT^\sfn U := \sfGsy^\sfn_{(0,\ldots,0)} U,
$$ 
is a {\em symmetric $n$-fold vector bundle} (\cite{GM12}), i.e., an edge-symmetric $n$-fold groupoid such that, for each edge, source and target projections coincide, and each vertex group carries moreover a $\K$-module structure.  
\end{theorem}

\begin{proof}
Theorem \ref{th:Sym2} implies that, for $\ttt = 0$, 
we have an edge-symmetric $n$-fold groupoid whose  source and target projections coincide, and 
the preceding Theorem \ref{th:Symscalar} implies that each vertex set
carries a $\K$-module structure which is invariant.
\end{proof}

\subsection{$C^{n,\sym}$-laws}

\begin{definition}
For $n \in \N$, or $n=\infty$,
a {\em $C^{n,\sym}$-law} between linear sets $U,U'$ is family of  morphism of $k$-fold groupoids, for
$k=0,\ldots,n$,  
$$
\bff^\sfk :
\sfGsy^\sfk U \to \sfGsy^\sfk U'
$$
which are compatible with each other: with notation from Equation (\ref{eqn:sym-n}), we have
$$
\forall k=1,\ldots,n: \qquad 
\pi_k \circ \bff^\sfk = (\bff^{\mathsf{k-1}} \times \id_\K) \circ \pi_k .
$$
Such a law is called {\em symmetric} if each $\bff^\sfk$ commutes with the $S_k$-action, and it is called 
 {\em homogeneous} if it commutes wih scalar actions:
$$
\forall \sss \in \K^n:\qquad
\bff^\sfn \circ \Phi_\sss^\alpha  =  \Phi_\sss^\alpha \circ \bff^\sfn .
$$
\end{definition}

Recall  from Definition \ref{th:symmetricformula} the higher order difference factorizer $f_\ttt^{[n]}$, and from Theorem
\ref{th:symmetricformula} its ``explicit formula''. Putting things together, we get explicit and ``closed'' formulae for all components of
$C^{n,\sym}$-laws. However, such formulae are not ``intrinsic'': in geometric terms, they depend on the flat affine connection of the
chart domain $U$.

\begin{theorem}\label{th:symdiff}
Let $\bff^\sfn$ be a $C^{n,\sym}$-law with base map $f:U \to U'$ and assume $\ttt \in (\K^\times)^n$. 
Then, for a general vertex $\alpha \in \cP(\sfn)$, 
the vertex map can be expressed in terms of the higher order difference factorizers of $f$ by
$$
\sfGsy^{\alpha;n}_{\ttt}  f (\vvv_\alpha) = \bigl( f_{\ttt_\beta}^{[\ell_\beta]} (\vvv_\beta ) \bigr)_{\beta \subset \alpha} ,
$$
where $\vvv_\lambda = (v_\mu)_{\mu \subset \lambda}$.
If, moreover,  $\K$ is a topological base ring and $f:U \to W$ is $C^{n,\sym}$ (Definition \ref{def:sym}), then 
the same formulae define a
 a continuous, symmetric and homogenous $C^{n,\sym}$-law $\bff^\sfn$, which is entirely determined by its base map $f$.
\end{theorem}

\begin{proof}
For  $n=1$, a $C^1$-law is explicitly described by its top-vertex map
$$
\sfG^{\sfone,\sfone}_{t_1} f
((v_0,v_1)) =  (f(v_0), f^{[1]}_{t_1}  (v_0,v_1,t_1)),
$$
Iterating this for $n=2$, we get 
\begin{equation*}
\sfGsy^{\sftwo,\sftwo}_{(t_1,t_2)} f
((v_0,v_1,v_2,v_{12})) =
\bigr( f(v_0), 
f^{[1]}_{t_1}(v_0,v_1),
f^{[1]}_{t_2}(v_0,v_2),
f^{[2]}_{t_1,t_2}(v_0,v_1,v_2,v_{12}) \bigl),
\end{equation*}
where $f^{[2]}$ is explicitly given by Equation (\ref{eqn:03}). 
For general $\sfn$ and $\alpha \in \cP(\sfn)$, the computation is the same, by induction.
Note that, on the finite part, this law is uniquely determined by the base map $f$, and moreover, it is symmetric,
since the explicit formulae obviously have a symmetric structure. 

\ssk
Now, in a topological setting,
if $f$ is $C^{n,\sym}$, then, ``by density'', all properties carry  over to arbitrary $\ttt$. Indeed, by definition of 
the $C^{n,\sym}$-property, the law can be defined by induction, and the explicit formulae show that it is continuous.
Symmetry now carries over by density, and so does homogeneity. 
\end{proof}

\subsection{Scalar extensions, and symmetry of polynomial laws}
Recall from Part I, Lemma 4.13, that $\sfG^{\sfone}_t$, applied to the ring product $\K \times \K\to \K$,
gives the ring $\K[X]/(X^2 - tX)$, together with its two projections to $\K$. Likewise, $\sfG^\sfone_t$ 
applying it to the (right) action map
$V \times \K \to V$ gives the $\K[X]/(X^2 - tX)$-module structure of the scalar extended module.

\begin{theorem}
Fix  $\ttt \in \K^n$. Applying $\sfGsy_\ttt^\sfn$ to the ring product $\K \times \K \to \K$ yields a {\em ring of type $\ul{\rm Goid}_n$},
that is, an $n$-fold groupoid $\sfG_\ttt^\sfn \K$, such that each vertex set
$\sfGsy^{\alpha;\sfn}_\ttt \K$ carries a ring structure, and both structures are compatible with each other:
the structure maps of the $n$-fold groupoid are morphisms of rings, and the collection of ring structures defines a morphism of
$n$-fold groupoids $\sfGsy^\sfn_\ttt (\K \times \K) \to \sfGsy_\ttt^\sfn \K$. 
\end{theorem}

\begin{proof}
For $n=1$, see the result from Part I quoted above; for $n\in \N$ the claim now follows by induction.
\end{proof}

\begin{theorem}
An algebraic model for the structures from the preceding theorem can be constructed as follows:
for $\alpha \in \cP(\sfn)$, consider the ring
\begin{align*}
\bA_\ttt^{\alpha;\sfn}&:=
\K[X]/(X^2 - t_{\alpha_1} X) \otimes \ldots \otimes \K[X]/(X^2 - t_{\alpha_\ell} X) 
\cr
& = 
 \K[X_{\alpha_1},\ldots,X_{\alpha_n}] / (X_{\alpha_1}^2 - t_{\alpha_1} X_{\alpha_1} , \ldots, X_{\alpha_\ell}^2 - t_{\alpha_\ell} X_{\alpha_\ell}) .
\end{align*} 
For every edge $(\beta,\alpha;\sfn)$, there are edge projections (source and target) and injections, such that
the whole structure forms a ring of type $\ul{\rm Goid}_n$, isomorphic to $\sfGsy_\ttt^\sfn \K$.
\end{theorem} 

\begin{proof}
Again, for $n=1$ this is shown in Part I, 4.4 and 4.5, and for $n\in \N$ it follows by induction.
\end{proof}

\begin{theorem}
Assume $P:\ul V \to \ul W$ is a polynomial law. 
Then $P$ defines a homogeneous symmetric $C^{\infty,\sym}$-law between $V$ and $W$. 
\end{theorem}

\begin{proof}
By induction it follows from Theorem 4.13 of Part I  that $P$ 
defines a homogeneous $C^{\infty,\sym}$-law between $V$ and $W$. 
It remains to prove that this law is {\em symmetric} (note that the ``density argument'', proving the corresponding fact in
topological differential calculus, is not available here!). 
Now, by definition these laws are constructed by scalar extension with the rings
$\bA_\ttt^\sfn = \sfGsy_\ttt^\sfn \K$, and the explicit formula shows that the group $S_n$ acts by automorphisms on
this ring of type $\ul{\rm Goid}_n$ (it permutes the $X_{\alpha_i}$ and at the same time the $t_{\alpha_i}$), and hence the laws
defined by scalar extension also commute with this action. 
\end{proof}

\section{Full cubic calculus}\label{sec:full}

The structure of the full cubic higher order groupoids is more complicated than the structure of the symmetric ones --
the main reason is that in symmetric calculus {\em scaleoids} $\sfGsy^\sfn 0$ have a trivial $n$-fold groupoid structure, whereas
in full calculus, complications are concentrated in the rather intricate structure of the scaleoids
$\sfG^\sfn 0$. 
In this chapter, we initiate the study of these higher order structures; but by no means we pretend that this study were complete.
Let us
start by describing in more detail the  algorithm allowing to compute  vertex sets and vertex maps.  To fix notation for this chapter,
let  $N = \{ a_1,\ldots, a_n\} \subset \N$ be a finite subset of cardinality $n$, and we assume that $a_1 < \ldots < a_n$; 
for a vertex  $\alpha \in \cP(N)$ of cardinality $k=\ell(\alpha)$, we write  $\alpha = \{ \alpha_1,\ldots,\alpha_k\}$, and assume
$\alpha_1 < \ldots < \alpha_k$.

\subsection{Top and bottom vertex sets}
First of all, recall from Theorem \ref{th:topvertex}  the description of the top and bottom vertex sets
\begin{align*}
\sfG^{N;N} U & = (\ldots (U^\sett{a_1})^\sett{a_2} \ldots )^\sett{a_n} = U^\sett{a_1,\ldots, a_n} = U^N,
\cr
\sfG^{\emptyset;N} U & = U \times \K^n = U \times 0^\sett{ a_1} \times \ldots \times 0^\sett{ a_n }.
\end{align*} 
For $N=\sfone$, $U^\sett{1} = \{ (v_0,v_1,t_1)\mid \, v_0 \in U, t_1 \in \K, v_0+t_1 v_1 \in U\}$, and
for $N = \sftwo$, see Equation (\ref{eqn:U-two}). 
By induction,  $U^N$ is a subset of $V^{2^n} \times \K^{2^n-1}$.
Its elements will be described by their ``affine coordinates'' 
\begin{equation}
(\vvv, \ttt) = 
\bigl( (v_\alpha)_{\alpha \in \cP(N)},(t_\alpha)_{\alpha \in \cP(N)\setminus \emptyset} \bigr),
\end{equation}
where the elements $v_\alpha \in V, t_\alpha \in \K$ satisfy $2^n$ inductively defined conditions
(the first of which is $v_\emptyset \in U$).

\subsection{Description of vertex sets and vertex maps for $\sfG^N$}
To define general vertex sets $\sfG^{\alpha;N} U$, 
the inductive procedure from Theorem \ref{th:Full1} (and
of Theorem \ref{th:nfoldcat})  leads us to distinguish, at each step of the induction procedure,  
 the cases of ``new'' and ``old'' vertices.
The following definition is designed  to deal with this distinction:

\begin{definition}
With $N$ and $\alpha \in \cP(N)$ as above,  for all $j \in N$, we define 
$$
\langle \alpha , j \rangle  := \Bigr\{
\begin{matrix}
\{ j \} & \mbox{ if } & j \in \alpha \cr 
\emptyset & \mbox{ if } & j\notin \alpha.
 \end{matrix}
 $$
\end{definition}

\begin{theorem}[Computing $\sfG^{\alpha;N}$]\label{th:fullcompute}
With notation as in the preceding definition, 
$$
\sfG^{\alpha;N} =
\sfG^{\langle \alpha, a_n \rangle, \{ a_n\} } \circ 
\sfG^{\langle \alpha, a_{n-1} \rangle, \{ a_{n-1}\} }  \circ \ldots \circ 
\sfG^{\langle \alpha, a_1 \rangle, \{ a_1 \} } .
$$
This applies both to the computation of $\sfG^{\alpha;N} U$ and of $\sfG^{\alpha;N} f$: 
under the assumptions of topological differential calculus, or for polynomial laws, the procedure permits to compute
the vertex maps $\sfG^{\alpha; N} f$  (Definition \ref{def:derived}) from the base map $f$. 
In particular, top vertex and bottom vertex maps are given by
\begin{align*}
\sfG^{N;N} f & = (\ldots (f^\sett{a_1})^\sett{a_2} \ldots )^\sett{a_n} =: f^\sett{a_1,\ldots, a_n},
\cr
\sfG^{\emptyset;N} f & = f \times \id_\K^n . 
\end{align*}
\end{theorem}

\begin{proof}
Written out in more detail, to compute $\sfG^{\alpha;N}U$, the procedure from Theorem \ref{th:Full1} says that,
with $N' = N \setminus \{ a_n \}$ and  $\alpha' = \alpha \setminus \{ a_k \}$,
\begin{enumerate}
\item
if the top element $a_n$ belongs to $\alpha$ (``new vertex'') then we ``derive'':
$$
\sfG^{\alpha;N} U = \sfG^{\{ a_kn\}; \{ a_n \}} (\sfG^{ \alpha' ; N'} U),
$$
\item
if $a_n \notin \alpha$ (``old vertex''), then we ``take a product with $\K$''
$$
\sfG^{\alpha;N} U = \sfG^{\alpha;N'} U \times \K = \sfG^{\alpha;N'} U \times_0 0^\sett{a_k} ,
$$
\end{enumerate}
and so on. This is exactly what the algorithm described above does.
\end{proof}

\begin{remark}
When ``deriving pullbacks'', i.e., 
 when computing the composition of symbols
$\sfG^{\{ i\} , \{ i \}} \circ \sfG^{\emptyset , \{ i-1 \}}$, one has to take account of Relation (\ref{eqn:pb!}).
For instance, the vertex $\{ 2 \}$ of the $2$-cube $\{ 1,2\}$ gives rise to the vertex set
$$
\sfG^{ \{ 2 \}, \{ 1,2 \}} U =
\sfG^{ \{ 2 \}, \{ 2 \} } \circ \sfG^{\emptyset , \{ 1 \}} U =
\sfG^{ \{ 2 \}, \{ 2 \} } ( U \times  0^\sett{1}  ) =
U^\sett{2} \times_{0^\sett{2}} 0^\sett{1,2},
$$
and since ``affine coordinates'' of $0^\sett{1,2}$ are $(t_1,t_2,t_{12}) \in \K^3$, we get those of
$\sfG^{ \{ 2 \}, \{ 1,2 \}} U$, namely $(v_0,v_2,t_1,t_2,t_{12})$ (where $t_2$ appears only once because we take the quotient by $0^\sett{2}$).
In the same way, we get the following tables  for
$N=\sfone $ and  $N=\sftwo$:
\end{remark}

\ssk

\nin \begin{tabular}{ c | l |  l  | l }
$N $ & $\alpha $ & vertex set $\sfG^{\alpha;N}U$ & affine coordinates of elements of $\sfG^{\alpha;N} U$
\\
\hline
$\sfone$ & $\emptyset$ & $U \times 0^\sett{1}$ & $(v_0,t_1)$ 
\\
 & $\{ 1 \}$ & $U^\sett{1}$ & $(v_0,v_1,t_1)$
\\
\hline
$\sftwo$ & $\emptyset$ & $U \times 0^\sett{1}  \times 0^\sett{2}$ & $(v_0,t_1,t_2)$
\\
 & $\{ 1 \}$ & $U^\sett{1}\times 0^\sett{2}$ & $(v_0,v_1,t_1,t_2)$
\\
 & $\{ 2 \}$ & $U^\sett{2}\times_{0^\sett{2} } 0^\sett{1,2} $ & $(v_0,v_2,t_1,t_2,t_{12})$
\\
 & $\{ 1,2 \}$ & $U^\sett{1,2}$ & $(v_0,v_1,v_2,v_{12},t_1,t_2,t_{12})$
\end{tabular}

\msk
\nin
To get the corresponding table of vertex sets
for $N= \mathsf{3}$, for the ``old'' vertices, juste  copy and paste the table for $N=\sftwo$
 and add everywhere $\times 0^\sett{3}$; for the
``new''  vertices, we  ``derive by $\sett{3}$'': everwhere $3$ is added in the superscripts, and
$\times_{0^A}$ is replaced by $\times_{0^{A \cup \{ 3\}}}$. This gives:

\msk
\nin \begin{tabular}{ c | l |  l  | l }
$N $ & $\alpha $ & vertex set $\sfG^{\alpha;N}U$ & affine coordinates of elements of $\sfG^{\alpha;N} U$
\\
\hline
$\sfthree$ & $\emptyset$ & $U \times 0^\sett{1} \times 0^\sett{2} \times 0^\sett{3}$ & $(v_0,t_1,t_2,t_3)$
\\
 & $\{ 1 \}$ & $U^\sett{1}\times  0^\sett{2} \times 0^\sett{3} $ & $(v_0,v_1,t_1,t_2,t_3)$
\\
 & $\{ 2 \}$ & $U^\sett{2}\times_{0^\sett{2}} 0^\sett{1,2} \times 0^\sett{3}$ & $(v_0,v_2,t_1,t_2,t_{12},t_3)$
\\
 & $\{ 1,2 \}$ & $U^\sett{1,2} \times 0^\sett{3}$ & $(v_0,v_1,v_2,v_{12},t_1,t_2,t_{12},t_3)$
\\
\hline
& $\{ 3 \}$ & $U^\sett{3} \times_{0^\sett{3}} 0^\sett{13} \times_{0^\sett{3}} 0^\sett{2,3}$ & 
$(v_0,v_3,t_1,t_3,t_{13},t_2,t_{23})$
\\
& $\{ 1,3 \}$ & $U^\sett{1,3} \times_{0^\sett{3}} 0^\sett{2,3}$ & $(v_0,v_1,v_3,v_{13}, t_1,t_3,t_2,t_{13},t_{23})$
\\
& $\{ 2,3 \}$ & $U^\sett{2,3} \times_{0^\sett{ 2,3 }} 0^\sett{1,2,3}$ &
$(v_0,v_2,v_3,v_{23}, t_1,t_2,t_3,t_{12},t_{13},t_{23},t_{123})$
\\
& $\{ 1, 2,3 \}$ & $U^\sett{1,2,3}$ &
$(v_0,v_1,v_2,v_3,v_{23},v_{13},v_{12},v_{123}, $
\\
& & & $\qquad \qquad 
t_1,t_2,t_3,t_{12},t_{13},t_{23},t_{123})$
\end{tabular}

\msk
\nin
The underlying combinatorial scheme is what we call the {\em $\alpha$-stair}, e.g., for $N=\sfthree$:

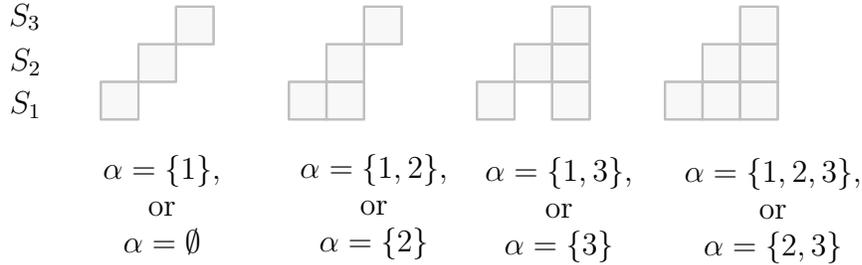
\begin{center}
\psset{xunit=0.5cm,yunit=0.5cm,algebraic=true,dimen=middle,dotstyle=o,dotsize=3pt 0,linewidth=0.8pt,arrowsize=3pt 2,arrowinset=0.25}
\begin{pspicture*}(-4.3,-4.08)(19.66,4.00)
\pspolygon[linecolor=lightgray,fillcolor=lightgray,fillstyle=solid,opacity=0.1](0.,0.)(0.,1.)(-1.,1.)(-1.,0.)
\pspolygon[linecolor=lightgray,fillcolor=lightgray,fillstyle=solid,opacity=0.1](1.,1.)(1.,2.)(0.,2.)(0.,1.)
\pspolygon[linecolor=lightgray,fillcolor=lightgray,fillstyle=solid,opacity=0.1](2.,2.)(2.,3.)(1.,3.)(1.,2.)
\pspolygon[linecolor=lightgray,fillcolor=lightgray,fillstyle=solid,opacity=0.1](5.,0.)(5.,1.)(4.,1.)(4.,0.)
\pspolygon[linecolor=lightgray,fillcolor=lightgray,fillstyle=solid,opacity=0.1](6.,0.)(6.,1.)(5.,1.)(5.,0.)
\pspolygon[linecolor=lightgray,fillcolor=lightgray,fillstyle=solid,opacity=0.1](6.,1.)(6.,2.)(5.,2.)(5.,1.)
\pspolygon[linecolor=lightgray,fillcolor=lightgray,fillstyle=solid,opacity=0.1](7.,2.)(7.,3.)(6.,3.)(6.,2.)
\pspolygon[linecolor=lightgray,fillcolor=lightgray,fillstyle=solid,opacity=0.1](10.,0.)(10.,1.)(9.,1.)(9.,0.)
\pspolygon[linecolor=lightgray,fillcolor=lightgray,fillstyle=solid,opacity=0.1](11.,1.)(11.,2.)(10.,2.)(10.,1.)
\pspolygon[linecolor=lightgray,fillcolor=lightgray,fillstyle=solid,opacity=0.1](12.,2.)(12.,3.)(11.,3.)(11.,2.)
\pspolygon[linecolor=lightgray,fillcolor=lightgray,fillstyle=solid,opacity=0.1](12.,1.)(12.,2.)(11.,2.)(11.,1.)
\pspolygon[linecolor=lightgray,fillcolor=lightgray,fillstyle=solid,opacity=0.1](12.,0.)(12.,1.)(11.,1.)(11.,0.)
\pspolygon[linecolor=lightgray,fillcolor=lightgray,fillstyle=solid,opacity=0.1](15.,0.)(15.,1.)(14.,1.)(14.,0.)
\pspolygon[linecolor=lightgray,fillcolor=lightgray,fillstyle=solid,opacity=0.1](16.,0.)(16.,1.)(15.,1.)(15.,0.)
\pspolygon[linecolor=lightgray,fillcolor=lightgray,fillstyle=solid,opacity=0.1](17.,0.)(17.,1.)(16.,1.)(16.,0.)
\pspolygon[linecolor=lightgray,fillcolor=lightgray,fillstyle=solid,opacity=0.1](16.,1.)(16.,2.)(15.,2.)(15.,1.)
\pspolygon[linecolor=lightgray,fillcolor=lightgray,fillstyle=solid,opacity=0.1](17.,1.)(17.,2.)(16.,2.)(16.,1.)
\pspolygon[linecolor=lightgray,fillcolor=lightgray,fillstyle=solid,opacity=0.1](17.,2.)(17.,3.)(16.,3.)(16.,2.)
\psline[linecolor=lightgray](0.,0.)(0.,1.)
\psline[linecolor=lightgray](0.,1.)(-1.,1.)
\psline[linecolor=lightgray](-1.,1.)(-1.,0.)
\psline[linecolor=lightgray](-1.,0.)(0.,0.)
\psline[linecolor=lightgray](1.,1.)(1.,2.)
\psline[linecolor=lightgray](1.,2.)(0.,2.)
\psline[linecolor=lightgray](0.,2.)(0.,1.)
\psline[linecolor=lightgray](0.,1.)(1.,1.)
\psline[linecolor=lightgray](2.,2.)(2.,3.)
\psline[linecolor=lightgray](2.,3.)(1.,3.)
\psline[linecolor=lightgray](1.,3.)(1.,2.)
\psline[linecolor=lightgray](1.,2.)(2.,2.)
\psline[linecolor=lightgray](5.,0.)(5.,1.)
\psline[linecolor=lightgray](5.,1.)(4.,1.)
\psline[linecolor=lightgray](4.,1.)(4.,0.)
\psline[linecolor=lightgray](4.,0.)(5.,0.)
\psline[linecolor=lightgray](6.,0.)(6.,1.)
\psline[linecolor=lightgray](6.,1.)(5.,1.)
\psline[linecolor=lightgray](5.,1.)(5.,0.)
\psline[linecolor=lightgray](5.,0.)(6.,0.)
\psline[linecolor=lightgray](6.,1.)(6.,2.)
\psline[linecolor=lightgray](6.,2.)(5.,2.)
\psline[linecolor=lightgray](5.,2.)(5.,1.)
\psline[linecolor=lightgray](5.,1.)(6.,1.)
\psline[linecolor=lightgray](7.,2.)(7.,3.)
\psline[linecolor=lightgray](7.,3.)(6.,3.)
\psline[linecolor=lightgray](6.,3.)(6.,2.)
\psline[linecolor=lightgray](6.,2.)(7.,2.)
\psline[linecolor=lightgray](10.,0.)(10.,1.)
\psline[linecolor=lightgray](10.,1.)(9.,1.)
\psline[linecolor=lightgray](9.,1.)(9.,0.)
\psline[linecolor=lightgray](9.,0.)(10.,0.)
\psline[linecolor=lightgray](11.,1.)(11.,2.)
\psline[linecolor=lightgray](11.,2.)(10.,2.)
\psline[linecolor=lightgray](10.,2.)(10.,1.)
\psline[linecolor=lightgray](10.,1.)(11.,1.)
\psline[linecolor=lightgray](12.,2.)(12.,3.)
\psline[linecolor=lightgray](12.,3.)(11.,3.)
\psline[linecolor=lightgray](11.,3.)(11.,2.)
\psline[linecolor=lightgray](11.,2.)(12.,2.)
\psline[linecolor=lightgray](12.,1.)(12.,2.)
\psline[linecolor=lightgray](12.,2.)(11.,2.)
\psline[linecolor=lightgray](11.,2.)(11.,1.)
\psline[linecolor=lightgray](11.,1.)(12.,1.)
\psline[linecolor=lightgray](12.,0.)(12.,1.)
\psline[linecolor=lightgray](12.,1.)(11.,1.)
\psline[linecolor=lightgray](11.,1.)(11.,0.)
\psline[linecolor=lightgray](11.,0.)(12.,0.)
\psline[linecolor=lightgray](15.,0.)(15.,1.)
\psline[linecolor=lightgray](15.,1.)(14.,1.)
\psline[linecolor=lightgray](14.,1.)(14.,0.)
\psline[linecolor=lightgray](14.,0.)(15.,0.)
\psline[linecolor=lightgray](16.,0.)(16.,1.)
\psline[linecolor=lightgray](16.,1.)(15.,1.)
\psline[linecolor=lightgray](15.,1.)(15.,0.)
\psline[linecolor=lightgray](15.,0.)(16.,0.)
\psline[linecolor=lightgray](17.,0.)(17.,1.)
\psline[linecolor=lightgray](17.,1.)(16.,1.)
\psline[linecolor=lightgray](16.,1.)(16.,0.)
\psline[linecolor=lightgray](16.,0.)(17.,0.)
\psline[linecolor=lightgray](16.,1.)(16.,2.)
\psline[linecolor=lightgray](16.,2.)(15.,2.)
\psline[linecolor=lightgray](15.,2.)(15.,1.)
\psline[linecolor=lightgray](15.,1.)(16.,1.)
\psline[linecolor=lightgray](17.,1.)(17.,2.)
\psline[linecolor=lightgray](17.,2.)(16.,2.)
\psline[linecolor=lightgray](16.,2.)(16.,1.)
\psline[linecolor=lightgray](16.,1.)(17.,1.)
\psline[linecolor=lightgray](17.,2.)(17.,3.)
\psline[linecolor=lightgray](17.,3.)(16.,3.)
\psline[linecolor=lightgray](16.,3.)(16.,2.)
\psline[linecolor=lightgray](16.,2.)(17.,2.)
\rput[lt](-0.94,-0.88){\parbox{2.2 cm}{$\begin{matrix}  \alpha = \{ 1 \}, \cr  \mbox{ or } \cr  \alpha = \emptyset  \end{matrix}$}}
\rput[tl](4.3,-0.86){\parbox{2.2 cm}{$\begin{matrix}  \alpha = \{ 1,2 \}, \cr  \mbox{ or } \cr  \alpha = \{ 2\}  \end{matrix}$}}
\rput[tl](9.22,-0.94){\parbox{2.2 cm}{$\begin{matrix}  \alpha = \{ 1,3 \}, \cr  \mbox{ or } \cr  \alpha = \{ 3 \}  \end{matrix}$}}
\rput[tl](14.52,-0.96){\parbox{2.2 cm}{$\begin{matrix}  \alpha = \{ 1,2,3 \}, \cr  \mbox{ or } \cr  \alpha = \{2,3 \} \end{matrix}$}}
\rput[tl](-3.42,0.74){$S_1$}
\rput[tl](-3.42,1.88){$S_2$}
\rput[tl](-3.44,3.08){$S_3$}
\end{pspicture*}
\end{center}

\begin{theorem}[Computing $\sfG^{\alpha,\sfn}$, bis]\label{th:fullcompute1}
Let $N = \{ a_1 , \ldots , a_n \} \subset \N$ (ordered),
$\alpha \subset N$ and define
the {\em $\alpha$-stair} $S=S(\alpha)$, given by the following $n$ subsets $S_1,\ldots,S_n \subset N$:
$$
S_i = \{ a_i \} \cup (\alpha \setminus \{ a_1,\ldots,a_{i-1} \} = \{ a_j \mid \, j=i \mbox{ or }
(j<i \mbox{ and } a_j \in \alpha ) \}.
$$
Then the vertex set $\sfG^{\alpha;N} U$ is given by
$$
\sfG^{\alpha;N} U=
U^\alpha \times_{0^{\alpha \setminus \{ a_1 \}}} 0^{S_1} \times_{0^{S_1 \cap S_2}} 0^{S_2} \times \ldots \times_{0^{S_{n-1}\cap S_n}}
0^{S_n} .
$$
In particular,  for $U = 0$, we get the
 vertex sets $\sfG^{\alpha;N} 0$ of a poit:
$$
\sfG^{\alpha;N} 0 =
0^\alpha \times_{0^{\alpha \setminus \{ a_1 \}}} 0^{S_1} \times_{0^{S_1 \cap S_2}} 0^{S_2} \times \ldots \times_{0^{S_{n-1}\cap S_n}}
0^{S_n}
$$
\end{theorem}

\begin{proof}
This description results in anlalyzing the construction from Theorem \ref{th:fullcompute}.
First of all, note that sets $\alpha,\alpha'$ that differ only by the element $1$ give rise to the same $\alpha$-stair.
For $\alpha = \emptyset$, the $\alpha$-stair is given by the singletons $S_i = \{ a_i \}$, and the theorem describes correctly
the construction of the bottom vertex set.
Now, for general $\alpha$, we have to take account of the occurences of ``derivation into $\alpha_i$-direction''.
When $\alpha_1 = a_1$, derivation concerns only the base set, so replace $U$ by $U^\sett{a_1}$;
but if $\alpha_1 > a_1$, then it concerns all preceding indices. (In our graphical scheme, this means that we fill up all
boxes over $\alpha_2$ until you reach the diagonal.) 
We continue this way for all indices up to $\alpha_k$ (fill up all boxes over the the $\alpha_i$ until you reach the diagonal).
\end{proof}

\begin{remark}
Of course, in the end, factors  may cancel out: whenever $S_{k+1}\subset S_k$, the corresponding term will
cancel. For instance, if $\alpha = N$ (the ``solid $\alpha$-stair''), then only the term $U^N$ remains  (top vertex).
\end{remark}

\begin{remark}
The theorem shows that for scaleoids ($U=0$), the vertex sets are uniquely determined by the $\alpha$-stair.
For $N = \sfone, \sftwo, \sfthree$, the following table is obtained from the preceding ones by taking $U= 0$;
for $N= \mathsf{4}$,  the reader may compute it herself, by drawing the $8$ possible $\alpha$-stairs :
\end{remark}

\ssk

\nin \begin{tabular}{ c | l |  l  | l }
$N $ & $\alpha $ & vertex set $\sfG^{\alpha;N}0$ & affine coordinates of elements of $\sfG^{\alpha;N} 0$
\\
\hline
$\sfone$
 & $\{ 1 \}$ & $0^\sett{1}$ & $(t_1)$
\\
\hline
$\sftwo$ &  $\{ 1 \}$ & $0^\sett{1} \times 0^\sett{2} $ & $(t_1,t_2)$
\\
 & $\{ 1,2 \}$ & $0^\sett{1,2}$ & $(t_1,t_2,t_{12})$
\\
\hline
$\sfthree$ &  $\{ 1 \}$ & $0^\sett{1}\times 0^\sett{2} \times 0^\sett{3}$ & $(t_1,t_2,t_3)$
\\
 & $\{ 1,2 \}$ & $0^\sett{1,2} \times 0^\sett{3}$ & $(t_1,t_2,t_{12},t_3)$
\\
& $\{ 1,3 \}$ & $0^\sett{1,3} \times_{0^\sett{3}} 0^\sett{2,3}$ & $( t_1,t_3,t_2,t_{13},t_{23})$
\\
& $\{ 1, 2,3 \}$ & $0^\sett{1,2,3}$ & $ (t_1,t_2,t_3,t_{12},t_{13},t_{23},t_{123})$
\\
\hline
$\mathsf{4}$ & 
$\{ 1 \}$ & $0^\sett{1}\times 0^\sett{2} \times 0^\sett{3} \times 0^\sett{4}$ & $(t_1,t_2,t_3,t_4)$
\\
 & $\{ 1,2 \}$ & $0^\sett{1,2} \times 0^\sett{3} \times 0^\sett{4}$ & $(t_1,t_2,t_{12},t_3,t_4)$
\\
& $\{ 1,3 \}$ & $0^\sett{1,3} \times_{0^\sett{3}} 0^\sett{2,3} \times 0^\sett{4}$ & $( t_1,t_3,t_2,t_{13},t_{23},t_4)$
\\
& $\{ 1, 2,3 \}$ & $0^\sett{1,2,3} \times 0^\sett{4}$ & $ (t_1,t_2,t_3,t_{12},t_{13},t_{23},t_{123},t_4)$
\\
 & $\{ 1,4 \}$ & $0^\sett{1,4}\times_{0^\sett{4}}   0^\sett{2,4} \times_{0^\sett{4}}  0^\sett{3,4} $ & $(t_1,t_2,t_3,t_4,t_{14},t_{24},t_{34})$
\\
 & $\{ 1,2,4 \}$ & $0^\sett{1,2,4} \times_{0^\sett{4}}  0^\sett{3,4}$ & $(t_1,t_2,t_3,t_4, t_{12},t_{24}, t_{14},t_{124},t_{34})$
\\
& $\{ 1,3,4 \}$ & $0^\sett{1,3,4} \times_{0^\sett{3,4}} 0^\sett{2,3,4}$ & $( t_1,t_2,t_3,t_4, t_{13},t_{14},t_{34},t_{134},t_{23},t_{24},t_{234})$
\\
& $\{ 1, 2,3,4 \}$ & $0^\sett{1,2,3,4}$ &
$( t_\alpha, \alpha \subset \sffour, \alpha \not= \emptyset) $
\end{tabular}

\ssk

\begin{remark}
For a $C^n$-map $f:U \to U'$, the vertex maps  $\sfG^\sfn f$ can be expressed,
with respect to affine coordinates, via  higher order difference factorizers $f^{[k]}$:
the analog of the formula from Theorem  \ref{th:symdiff} is
$$
\sfG^{\alpha;n} f (\vvv_\alpha , \ttt_\alpha) = \Bigl( \bigl( f^{[\ell_\beta]} (\vvv_\beta , \ttt_\beta ) \bigr)_{\beta \subset \alpha} ,
\ttt_\alpha \Bigr)  ,
$$
where $\vvv_\lambda = (v_\mu)_{\mu \subset \lambda}$, and now $\ttt_\lambda = (t_\mu)_{\mu \subset \lambda}$ has
also components for $\vert \mu \vert  > 1$.
\end{remark}

\subsection{Source projections, unit sections, and composition law}
These structural data of $\sfG^\sfn U$ are easy to describe:
fixing an edge $(\beta,\alpha;N)$ of the $n$-cube $\cP(N)$, the vertex sets 
$\sfG^{\beta;N}U$ and $\sfG^{\alpha;N}U$ have been described above. In affine coordinates, there is a canonical inclusion
$\sfG^{\beta;N}U \subset \sfG^{\alpha;N}U$, which describes the unit (or zero) section
$z^{\beta,\alpha;N}$ (by induction from the inclusion
$U \times \K \to U^\sett{1}$,
$(v_0,t_1) \mapsto (v_0,0,t_1)$), and there are natural projections, which describe the source projection
$\pi_0^{\beta,\alpha;N}$ 
(by induction from $\pi_0(v_0,v_1,t_1)=(v_0,t_1)$).
The inductive proofs are almost trivial since all maps in question are $\K$-linear, and their derivative is linear and of the same
form (cf.\ Part I, Section 4.3). For the same reason, since at $n=1$, the composition law
$$
(v_0',v_1',t_1) \ast (v_0,v_1,t_1) = (v_0,v_1' + v_1,t_1)
$$
is linear, the composition law corresponding to $(\beta,\alpha;N)$ will again be linear, and basically given by the same formula:
addition of certain components. More precisely, we get, in affine coordinates, if $\alpha = \beta \cup \{ i \}$,
\begin{align}
(\vvv',\ttt) \ast_{\beta,\alpha;N} (\vvv,\ttt) & = (\www,\ttt),
\cr
w_\gamma & = 
v_\gamma' + v_\gamma \mbox{ if } \gamma \subset \alpha, i \in \gamma, \quad
w_\gamma = v_\gamma \mbox{ else.}
\end{align}
The condition describing for which $\vvv',\vvv$ this composition is actually defined, depends of course on the
target projections.
These are much harder to describe, since for $n=1$, the target projection
$\pi_1(v_0,v_1,t_1)=(v_0 + t_1 v_1,t_1)$ is quadratic, and by deriving we get higher and higher degree polynomials
(cf.\ Part I, Theorem 4.10: the formula for $f^\sett{1}$ given there contains a trilinear term).

\subsection{Target projections corresponding to an edge $(\beta,\alpha;N)$}
Recall the ``way of talking'' from Appendix \ref{app:hypercube}: in induction procedures invoking edges,
we distinguish 3 kinds of edges, called ``old'', ``copy of old'' and ``new''.

\begin{theorem}[Edge projections]\label{th:fullcompute2}
Let $N\subset \N$ be a finite set,  $\alpha \in \cP(N)$ be a vertex, $\alpha = \{ \alpha_1,\ldots,\alpha_k\}$ with
$\alpha_1 < \ldots < \alpha_k$, and let $(\beta,\alpha)$ be an edge, so $\beta = \alpha \setminus \{ \alpha_i \}$ for some
$i \in \{ 1,\ldots , k\}$, and let $U \subset V$ be a linear set.
The target edge projections $\pi^{\beta,\alpha;N}_U:\sfG^{\alpha;N}U \to \sfG^{\beta,N}U$ are polynomial and
can be computed inductively as follows: for $n=1$, see Section 2, and for $n>1$, 
\begin{enumerate}
\item
(``old edge'')
if $\alpha_i \leq  \alpha_k < a_n$ (the top element belongs neither to $\alpha$ nor to $\beta$), then
$$
\pi^{\beta,\alpha;N} = \pi^{\beta,\alpha;N'}_U  \times \id_\K = \pi^{\beta,\alpha;N'} \times \id_{0^\sett{a_k}} , 
$$
\item
(``copy of old edge'')
if $\alpha_i < \alpha_k = a_n$ (the top element belongs to $\beta$ and $\alpha$), then
$$
\pi^{\beta,\alpha;N}_U = \sfG^{\{ a_k \}; \{ a_k \}} \pi^{\beta',\alpha';N'}_U
$$
where $\alpha' =\alpha \setminus \{ \alpha_k \} ,\beta' = \beta \setminus \{ \alpha_k \} , N' = N \setminus \{ \alpha_k \}$,
\item
(``new edge'')
if $\alpha_i = \alpha_k = a_n$ (the top element belongs to $\alpha$, but not to $\beta$), then
$\pi^{\beta,\alpha;N}$ is the base projection of the groupoid 
with morphism set $G^{\alpha,N} U$ and object set $G^{\alpha',N'} U$, ie.
$$
\pi^{\beta,\alpha;N} = \pi^{\emptyset,\{a_n\};\{a_n\}}_{\sfG^{\alpha';N'}U} .
$$
\end{enumerate}
\end{theorem}

\begin{proof}
The proof amounts to analyzing the inductive construction from Theorem \ref{th:Full1}: with $N =\sfn$,
Cases (1) and (2) come from deriving the family of projections $\sfG^{\mathsf{n-1}}U \to \sfG^{\mathsf{n-2}} U$ 
with respect to the new top element $a_n$, according to the procedure described above (distinguishing two cases, at each step),
and Case (3) comes from the very definition of the groupoid $\sfG^{\mathsf{n}}U \downdownarrows \sfG^{\mathsf{n-1}} U \times \K$. 
\end{proof}

\nin
Here is the list of the target projections in affine coordinates for $N = \sfone$ and
$N = \sftwo$:

\msk

\nin \begin{tabular}{ c | l |     l  }
$N $ & edge $(\beta,\alpha) $ &  coordinate formula of the edge projection $\pi = \pi^{ \beta,\alpha;N}_1$ 
\\
\hline
$\sfone$ & $(\emptyset,\{ 1\})$ & $\pi(v_0,v_1,t_1) =  (v_0 + t_1 v_1,t_1)$
\\
\hline
$\sftwo$  & $(\emptyset, \{ 1 \})$ & $\pi (v_0,v_1,t_1,t_2 ) = (v_0 + t_1 v_1, t_1, t_2)$
\\
& $(\{ 2 \}, \{ 1,2 \})$ & $\pi(v_0,v_1,v_2,v_{12},t_1,t_2,t_{12}) =$
\\
& & $\qquad \qquad (v_0 + t_1 v_1 , v_2 + t_{12} v_1 + t_1 v_{12} + 
t_2 t_{12} v_{12}, t_1,t_2, t_{12})$
\\
& $(\emptyset, \{ 2\} )$ & $\pi(v_0,v_2,t_1,t_2,t_{12}) = (v_0 + t_2 v_2, t_1 + t_{12} t_2, t_1)$
\\
& $(\{ 1 \}, \{ 1,2 \})$ & $\pi(v_0,v_1,v_2,v_{12},t_1,t_2,t_{12}) = (v_0 + t_2 v_2,v_1 + t_2 v_{12}, t_1 + t_2 t_{12},t_2)$
\end{tabular}

\msk
\nin
For $N=\sftwo$, the first line is Case (1), the second Case (2), and the following two Case (3).
Observe that the projection in 
Case (2) is a polynomial of degree 3, since it arises by ``deriving'' the projection from $N=\sfone$; the other projections
are quadratic.
For $N=\sfthree$, the projection 
 belonging to the edge $(\{ 2,3 \}, \sfthree)$
is the ``copy of the copy'', and will be of degree $5$. 
However, in case of a scaleoid ($U=0$, so all $v_\alpha = 0$), the fomulae simplify, and deriving once more we get:

\msk
\nin \begin{tabular}{ c | l |     l  }
$N $ & edge $(\beta,\alpha) $ &  coordinate formula of  the edge projection $\pi = \pi^{ \beta,\alpha;N}_1$, for $U=0$
\\
\hline
$\sfone$ &  none  &  trivial groupoid
\\
\hline
$\sftwo$ & $(\{ 1 \}, \{ 1,2 \})$ & $\pi(t_1,t_2,t_{12}) = ( t_1 + t_2 t_{12},t_2)$
\\
\hline
$\sfthree$  & 
$(\{ 1 \} , \{ 1,2 \} )$ & $\pi(t_1,t_2,t_3,t_{12}) = (t_1 + t_2 t_{12},t_2,t_3)$
\\
& $(\{  1,3 \}, \{ 1,2, 3 \} )$ & $\pi(  t_1,t_2,t_3,t_{12},t_{13},t_{23}, t_{123}) =$
\\
& & \qquad
$(t_1 + t_2 t_{12} , t_2, t_3, t_{13} + t_2 t_{123} + t_{23}t_{12}+ t_3 t_2 t_{12}, t_{23})$
\\
& $(\{ 1, 2  \}, \{ 1,2,3 \} )$ & $\pi ( t_1,t_2,t_3,t_{12},t_{13},t_{23}, t_{123}) =
(t_1 + t_3 t_{13}, t_2 + t_3 t_{23}, t_3, t_{12}+t_3 t_{123})$
\\
& $(\{ 1 \} , \{ 1,3 \})$ & $\pi(t_1,t_2,t_3,t_{13},t_{23}) = (t_1 + t_3 t_{13}, t_2,t_3)$
\end{tabular}

\msk
\nin
It is certainly possible to get a general combinatorial formula for general target projections, but it will be very complicated
since the number of terms explodes.

\begin{theorem}[Imbedding of $\sfGsy^\sfn$ into $\sfG^\sfn$]
The set $\K^n = 0^\sett{1}\times \ldots \times 0^\sett{n}$ 
is a subgroupoid of $\sfG^\sfn 0$, and it carries a trivial $n$-fold groupoid structure.
The inverse image of $\K^n$ under the projection
$\sfG^\sfn U \to \sfG^\sfn 0$ is the sub-$n$-fold groupoid $\sfGsy^\sfn U$.
\end{theorem}

\begin{proof}
For $\sfn = \sfone, \sftwo$, it can be read from the preceding table that the set defined by the conditions
``$t_\gamma = 0$ whenever $\vert \gamma \vert >1$'' defines a trivial subgroupoid of $\sfG^\sfn 0$.
For general $n$, the claim follows by induction; and more precisely, the induction procedure consists at each step by
applying the functor $\sfG^\sett{k}_{t_k}$, that is, exactly by the procedure defining the symbol $\sfGsy^\sfn$. On the level
of $U=0$, this just amounts to take a cartesian product with $\K$, at each step. 
\end{proof}

\subsection{The $2n$-fold category $\sfG^{\ol N} U$}\label{ssec:2n-fold}
To describe the $2n$-fold cat $\sfG^{\ol N} U$, several of the preceding results generalize, since 
the formal rules are similar:
from Definition \ref{def:pullback},
$$
\sfG^{\ol \sfone}  (A \times B) = \sfG^{\ol\sfone} A \times_{\sfG^{\ol \sfone}0 } \sfG^{\ol \sfone} B .
$$
To compute vertex sets,  recall from Definition \ref{def:G2} the vertex sets of $\sfG^{\ol \sfone} U$.
By induction, we see that the bottom vertex set $\sfG^{\ol N}U$ is the same as of $\sfG^N U$, namely $U \times \K^n$.
By induction, the top vertex set $U^{\ol N} = \sfG^{\ol N;\ol N} U$ has affine coordinates
$$
(v_\alpha , \sss_\beta , \ttt_\beta)_{ \alpha \in \cP(N), \beta \in \cP(N), \beta \not= \emptyset} \quad 
\in V^{2^n} \times \K^{2^n -1} \times \K^{2^n -1},
$$
that is,  the $\ttt$-variables are ``doubled''.
Next, consider the  types of vertices from Definition \ref{def:vertextype}:

\begin{theorem}\label{th:twotyped}
The small $2n$-fold cat $\sfG^{\ol N} U$ has the following $n$-fold subcats:
\begin{enumerate}
\item
the $N$-vertices define an $n$-fold subgroupoid isomorphic to $\sfG^N U$,
\item
the $n$-fold subcat of $\sfG^{\ol N} U$ having $N'$-vertices is naturally isomorphic to the product of the $n$-fold scaled action
category of $\K$ on $\K$ (Theorem \ref{th:SA}, where $S=K=\K$)  with the trivial cat on $U$.
\end{enumerate}
\end{theorem}

\begin{proof}
(1) 
Since $\sfG^N U$ is injected via a zero section into $\sfG^{\ol N} U$ as a subgroupoid, the vertex sets
$\sfG^{\alpha;\ol N} U$ and $\sfG^{\alpha;N} U$ corresponding to $N$-vertices can be identified.

(2)
For $n=1$, the cat $U \times \K^2 \rightrightarrows U \times \K$ (corresponding to the edge $(\emptyset, \{ 1' \})$)
is a direct product of the trivial cat on $U$ with the left action cat of $\K$ on
itself (Part I, Section 3.2). 
For $n=2$, we get a double cat $\sfG^{\sftwo' , \ol \sftwo} U$, with all projections of type ``$\partial$'':
$$
\begin{matrix} U \times \K^4 & \rightrightarrows & U \times \K^3 \cr
\downdownarrows & & \downdownarrows \cr
U \times \K^3 & \rightrightarrows & U \times \K^2 .
\end{matrix}
$$
Here, the target projections are the usual projections on cartesian products, but,
for instance, the two source projections starting at the top left vertex are
$\partial (v_0; s_1,t_1,s_2,t_2) = (v_0; s_1 t_1, s_2,t_2)$ and  $(v_0;s_1,t_1,s_2 t_2)$. 
For general $n$, the claim follows by induction, comparing the iteration procedure defining $\sfG^{\ol N}$ with the one from 
Theorem \ref{th:SA}.
\end{proof}

\begin{theorem}
For saturated vertices $\ol \alpha \subset \ol N$ (so $\alpha \subset N$), vertex sets can be computed by the algorithm described in
Theorem \ref{th:fullcompute}: letting,  for all $j \in N$, 
$$
\langle \alpha , j \rangle  := \Bigr\{
\begin{matrix}
\{ j , \ol j\} & \mbox{ if } & j \in \alpha \cr 
\emptyset & \mbox{ if } & j\notin \alpha,
 \end{matrix}
 $$
we have
$$
\sfG^{\ol \alpha;\ol N} =
\sfG^{\langle \alpha, a_n \rangle, \{ a_n, \ol a_n \} } \circ 
\sfG^{\langle \alpha, a_{n-1} \rangle, \{ a_{n-1}, \ol a_{n-1} \} }  \circ \ldots \circ 
\sfG^{\langle \alpha, a_1 \rangle, \{ a_1, \ol a_1 \} } .
$$
Likewise, the description from Theorem \ref{th:fullcompute1}
applies, with $\alpha$ and all $S_i$ overlined.
\end{theorem}

\begin{proof}
One uses the same arguments, where now the cat-rule $\sfG^\sett{k}$ is replaced by the doublecat-rule $\sfG^\settt{k}$.
\end{proof}

For generic vertices, the description of the vertex sets is more involved. 
Note also that, in contrast to the $\ast$-laws, the description of the interated $\bullet$-laws in affine coordinates will become as
complicated as the one of target projections, for the same reason: 
for $n=1$, $\bullet$ is bilinear, hence by deriving (applying copies of $\sfG^\sett{k}$), we get polynomial maps of higher and higher degree.
In a certain sense, it seems that the $2n$-fold cat rule $\sfG^{\ol \sfn}$ is ``assembled'' from the two $n$-fold cat rules
$\sfG^\sfn$ and $\sfA^{\sfn'}$, but at present it is not yet clear to me what kind of ``$n$-dimensional assembly construction'' is at work here.


\appendix 


\section{Natural hypercubes}\label{sec:hypercube}\label{app:hypercube}\label{App:hypercubes}

\subsection{Hypercubes}
Let $N$ be a finite set of cardinality $n:= \vert N \vert  \in \N$.  
By definition, the {\em hypercube of type $N$} is given by its {\em vertex set}, the power set $\cP(N)$, together with its natural
{\em partial order} given by inclusion.
The vertex $\emptyset$ is called the {\em bottom vertex}, and $N$ is called the {\em top vertex}.
An {\em (oriented) edge of $\cP(N)$} is a pair $(\beta,\alpha)$ of vertices such that $\alpha$ contains exactly one element more
than $\beta$, so
$\beta \subset \alpha$ and 
 $\vert \alpha \setminus \beta \vert = 1$.
 If $\alpha = \beta \cup \{ i \}$, we also say that $(\beta,\alpha)$ is {\em the edge from the vertex $\beta$ in direction $i$}.
More generally, for $k\leq n$, a {\em $k$-cube in $\cP(N)$} is given by a pair $(\beta,\alpha)$ of vertices such that
$\beta \subset \alpha$ and $\vert \alpha \setminus \beta \vert = k$.
For $k=2$, it is also called a {\em face}, for $k=3$ a {\em cube} and for $k=4$ a {\em tesseract}.

\subsection{Boolean ring}
Recall that $\cP(N)$ with the two operations
\begin{equation}
\alpha \cdot \beta = \alpha \cap \beta, \qquad
\alpha + \beta = (\alpha \cup \beta) \setminus (\alpha \cap \beta)
\end{equation}
forms a commutative ring such that $\alpha + \alpha = \emptyset$ and $\alpha \cdot \alpha =\alpha$ for all $\alpha$
(Boolean ring).  
It also carries a metric defined by $d(\alpha,\beta) = \vert \alpha + \beta \vert$.

\subsection{The natural $n$-hypercube}
By definition, the {\em natural $n$-hypercube} is the hybercube $\cP(\sfn)=\cP(\{ 1,\ldots, n \})$.
It has, as additional feature to the ones described above, a {\em total order on the vertex set}, given by
{\em lexicographic order} denoted by $\leq$:
\begin{equation}
\cP(\sfn) = \Bigl( \emptyset , \{ 1 \},\{ 2 \}, \{ 1,2 \}, \{ 3\}, \{ 1,3\}, \ldots , \sfn \Bigr)
\end{equation}
(which is the natural order of the ``binary codes'').
This total order is compatible with the partial order: $\beta \subset \alpha \Rightarrow \beta \leq \alpha$. 
More generally, such total order exists whenever $N$ is a finite subset of $\N$, with its order induced by the natural order of $\N$.
We then use the notation
$(\alpha;N)$ for vertices and $(\beta,\alpha;N)$ for oriented edges of $\cP(N)$.
When writing 
$\alpha = \{ \alpha_1,\ldots, \alpha_\ell \}$, we  assume that
$\alpha_1 < \ldots < \alpha_\ell$, and the cardinality
$\ell = \ell_\alpha$ will also be called the {\em length of $\alpha$}.

\begin{remark}
By elementary combinatorics it is seen that the number of $k$-cubes in an $n$-cube is ${n \choose k} 2^{n-k}$.
Thus there are $2^n$ vertices, $n \, 2^{n-1}$ edges and $n(n-1) 2^{n-3}$ faces. 
For instance, 
a tesseract has $16$ vertices, $32$ edges, $24$ faces, and $8$ cubes:

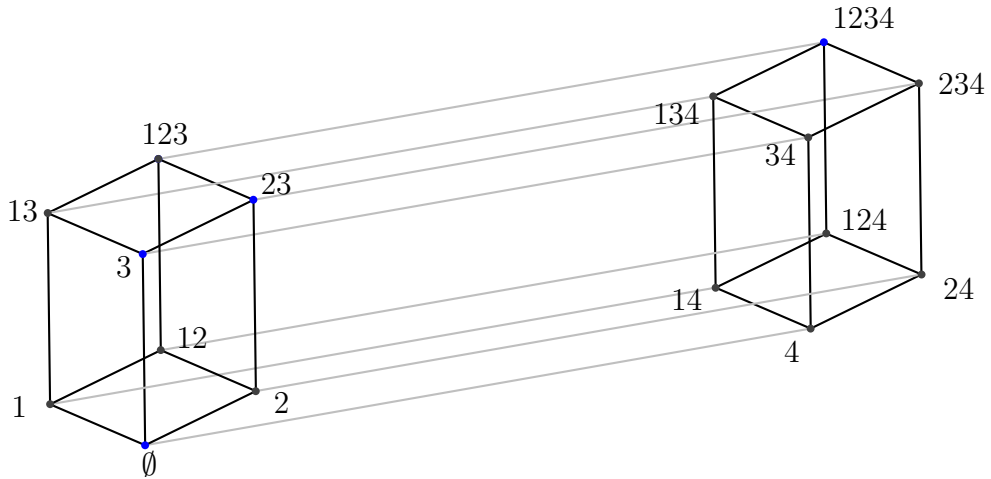
\begin{center}
\psset{xunit=0.8cm,yunit=0.8cm,algebraic=true,dotstyle=o,dotsize=3pt 0,linewidth=0.8pt,arrowsize=3pt 2,arrowinset=0.25}
\begin{pspicture*}(-0.92,-3.74)(19.84,4.06)
\psline[linecolor=lightgray](14.42,3.48)(3.36,1.54)
\psline(3.36,1.54)(4.94,0.86)
\psline(4.94,0.86)(3.1,-0.04)
\psline(3.1,-0.04)(3.14,-3.22)
\psline(3.36,1.54)(1.52,0.64)
\psline(1.52,0.64)(3.1,-0.04)
\psline[linecolor=lightgray](1.52,0.64)(12.58,2.58)
\psline(14.42,3.48)(12.58,2.58)
\psline(1.52,0.64)(1.56,-2.54)
\psline(1.56,-2.54)(3.14,-3.22)
\psline(14.42,3.48)(16,2.8)
\psline(16,2.8)(14.16,1.9)
\psline[linecolor=lightgray](3.1,-0.04)(14.16,1.9)
\psline(12.58,2.58)(14.16,1.9)
\psline[linecolor=lightgray](16,2.8)(4.94,0.86)
\psline(12.58,2.58)(12.62,-0.6)
\psline(14.16,1.9)(14.2,-1.28)
\psline(12.62,-0.6)(14.2,-1.28)
\psline[linecolor=lightgray](12.62,-0.6)(1.56,-2.54)
\psline[linecolor=lightgray](3.14,-3.22)(14.2,-1.28)
\psline(14.42,3.48)(14.46,0.3)
\psline(14.46,0.3)(12.62,-0.6)
\psline(16,2.8)(16.04,-0.38)
\psline(16.04,-0.38)(14.2,-1.28)
\psline(16.04,-0.38)(14.46,0.3)
\psline[linecolor=lightgray](3.4,-1.64)(14.46,0.3)
\psline(3.4,-1.64)(4.98,-2.32)
\psline[linecolor=lightgray](4.98,-2.32)(16.04,-0.38)
\psline(4.98,-2.32)(3.14,-3.22)
\psline(3.4,-1.64)(1.56,-2.54)
\psline(4.94,0.86)(4.98,-2.32)
\psline(3.36,1.54)(3.4,-1.64)
\rput[tl](3.08,-3.32){$ \emptyset $}
\rput[tl](0.92,-2.4){$ 1 $}
\rput[tl](5.28,-2.34){2}
\rput[tl](2.66,-0.08){3}
\rput[tl](0.84,0.86){13}
\rput[tl](5.06,1.3){23}
\rput[tl](3.08,2.14){123}
\rput[tl](13.76,-1.5){4}
\rput[tl](16.4,-0.44){24}
\rput[tl](16.32,2.96){234}
\rput[tl](14.56,4.06){1234}
\rput[tl](11.88,-0.62){14}
\rput[tl](3.66,-1.26){12}
\rput[tl](14.7,0.7){124}
\rput[tl](11.58,2.46){134}
\rput[tl](13.44,1.78){34}
\begin{scriptsize}
\psdots[dotstyle=*,linecolor=blue](14.42,3.48)
\psdots[dotstyle=*,linecolor=blue](3.36,1.54)
\psdots[dotstyle=*,linecolor=blue](4.94,0.86)
\psdots[dotstyle=*,linecolor=blue](3.1,-0.04)
\psdots[dotstyle=*,linecolor=blue](3.14,-3.22)
\psdots[dotstyle=*,linecolor=darkgray](1.52,0.64)
\psdots[dotstyle=*,linecolor=darkgray](12.58,2.58)
\psdots[dotstyle=*,linecolor=darkgray](1.56,-2.54)
\psdots[dotstyle=*,linecolor=darkgray](14.16,1.9)
\psdots[dotstyle=*,linecolor=darkgray](3.36,1.54)
\psdots[dotstyle=*,linecolor=darkgray](16,2.8)
\psdots[dotstyle=*,linecolor=darkgray](12.62,-0.6)
\psdots[dotstyle=*,linecolor=darkgray](14.2,-1.28)
\psdots[dotstyle=*,linecolor=darkgray](14.46,0.3)
\psdots[dotstyle=*,linecolor=darkgray](16.04,-0.38)
\psdots[dotstyle=*,linecolor=darkgray](3.4,-1.64)
\psdots[dotstyle=*,linecolor=darkgray](4.98,-2.32)
\end{scriptsize}
\end{pspicture*}
\end{center}

\nin
In the figure, vertices are labelled by $ij$, to abbreviate  $\{i,j\}$, etc.
\end{remark}

\subsection{A way of talking}\label{subsec:talking}
For induction procedures, a useful way of talking is

\begin{definition}
A vertex $\alpha$ of the $n+1$-hypercube is called
\begin{enumerate}
\item
an {\em old vertex} if it belongs to $\cP(\sfn)$, i.e., if $n+1 \notin \alpha$,
\item
a {\em new vertex} if $n+1 \in \alpha$.
\end{enumerate}
We distinguish three kinds of edges of $\cP(\mathsf{n+1})$:
an (oriented) edge  $(\beta,\alpha)$ is called
\begin{enumerate}
\item
an {\em old edge}  if both $\alpha$ and $\beta$ are old vertices;
\item
 a {\em copy of an old edge}  if both $\alpha$ and $\beta$ are new vertices (namely, it is  a copy of the ``old'' edge
$(\alpha \setminus \{ n+1\}, \beta \setminus \{ n+1 \}$)), and
\item
  a {\em new edge}  if 
$\beta$ is an old vertex and $\alpha$ is a new one (so that $\alpha = \beta \cup \{ n+1 \}$).
\end{enumerate}
Similarly, we distinguish three types of faces of $\cP(\mathsf{n+1})$: ``old'' ones, ``copies of old'' ones, and ``new'' ones. 
\end{definition}

\nin The  figure of a tesseract given above illustrates this for $n=3$: 
the ``old'' cube is the lower right cube $\cP(\sfthree)$;
edges are represented by segments joining the corresponding vertices (but the figure is just meant to represent the combinatorical
structure): new edges are represented by grey diagonal segments; old eges are
edges of the lower left cube and copies of old edges, of the upper right cube.

\subsection{The two-typed hypercube} This is another way of seeing the $2n$-hypercube.
for $n=2$ (tesseract), it is illustrated like this:

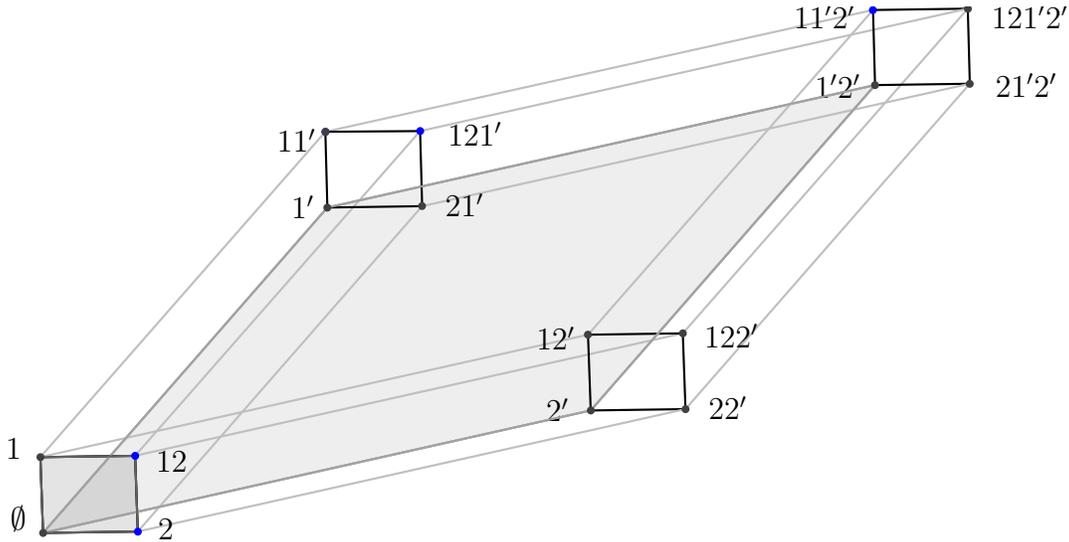
\begin{center}
\newrgbcolor{aqaqaq}{0.63 0.63 0.63}
\newrgbcolor{wqwqwq}{0.38 0.38 0.38}
\psset{xunit=0.9cm,yunit=0.9cm,algebraic=true,dotstyle=o,dotsize=3pt 0,linewidth=0.8pt,arrowsize=3pt 2,arrowinset=0.25}
\begin{pspicture*}(0.50,-3.87)(17.9,4.21)
\pspolygon[linecolor=aqaqaq,fillcolor=aqaqaq,fillstyle=solid,opacity=0.18](1.3,-3.57)(5.5,1.24)(13.59,3.05)(9.39,-1.76)
\pspolygon[linecolor=wqwqwq,fillcolor=wqwqwq,fillstyle=solid,opacity=0.17](1.3,-3.57)(1.26,-2.45)(2.66,-2.43)(2.7,-3.55)
\psline[linecolor=lightgray](13.56,4.16)(5.47,2.36)
\psline(5.47,2.36)(6.87,2.37)
\psline[linecolor=lightgray](6.87,2.37)(2.66,-2.43)
\psline(2.66,-2.43)(2.7,-3.55)
\psline[linecolor=lightgray](5.47,2.36)(1.26,-2.45)
\psline(1.26,-2.45)(2.66,-2.43)
\psline[linecolor=lightgray](1.26,-2.45)(9.35,-0.64)
\psline[linecolor=lightgray](13.56,4.16)(9.35,-0.64)
\psline(1.26,-2.45)(1.3,-3.57)
\psline(1.3,-3.57)(2.7,-3.55)
\psline(13.56,4.16)(14.96,4.18)
\psline[linecolor=lightgray](14.96,4.18)(10.75,-0.62)
\psline[linecolor=lightgray](2.66,-2.43)(10.75,-0.62)
\psline(9.35,-0.64)(10.75,-0.62)
\psline[linecolor=lightgray](14.96,4.18)(6.87,2.37)
\psline(9.35,-0.64)(9.39,-1.76)
\psline(10.75,-0.62)(10.79,-1.74)
\psline(9.39,-1.76)(10.79,-1.74)
\psline[linecolor=lightgray](9.39,-1.76)(1.3,-3.57)
\psline[linecolor=lightgray](2.7,-3.55)(10.79,-1.74)
\psline(13.56,4.16)(13.59,3.05)
\psline[linecolor=lightgray](13.59,3.05)(9.39,-1.76)
\psline(14.96,4.18)(14.99,3.07)
\psline[linecolor=lightgray](14.99,3.07)(10.79,-1.74)
\psline(14.99,3.07)(13.59,3.05)
\psline[linecolor=lightgray](5.5,1.24)(13.59,3.05)
\psline(5.5,1.24)(6.9,1.26)
\psline[linecolor=lightgray](6.9,1.26)(14.99,3.07)
\psline[linecolor=lightgray](6.9,1.26)(2.7,-3.55)
\psline[linecolor=lightgray](5.5,1.24)(1.3,-3.57)
\psline(6.87,2.37)(6.9,1.26)
\psline(5.47,2.36)(5.5,1.24)
\rput[tl](0.82,-3.16){$ \emptyset $}
\rput[tl](0.75,-2.17){$1 $}
\rput[tl](3,-3.37){2}
\rput[tl](4.97,1.42){$1'$}
\rput[tl](4.76,2.41){$11'$}
\rput[tl](7.24,1.45){$21'$}
\rput[tl](7.28,2.5){$121'$}
\rput[tl](11.13,-1.55){$22'$}
\rput[tl](8.73,-1.58){$2'$}
\rput[tl](15.37,3.22){$21'2'$}
\rput[tl](15.31,4.24){$121'2'$}
\rput[tl](8.59,-0.52){$12'$}
\rput[tl](2.97,-2.36){$12$}
\rput[tl](11.07,-0.48){$122'$}
\rput[tl](12.7,3.21){$1'2'$}
\rput[tl](12.39,4.24){$11'2'$}
\psline[linecolor=aqaqaq](1.3,-3.57)(5.5,1.24)
\psline[linecolor=aqaqaq](5.5,1.24)(13.59,3.05)
\psline[linecolor=aqaqaq](13.59,3.05)(9.39,-1.76)
\psline[linecolor=aqaqaq](9.39,-1.76)(1.3,-3.57)
\psline[linecolor=wqwqwq](1.3,-3.57)(1.26,-2.45)
\psline[linecolor=wqwqwq](1.26,-2.45)(2.66,-2.43)
\psline[linecolor=wqwqwq](2.66,-2.43)(2.7,-3.55)
\psline[linecolor=wqwqwq](2.7,-3.55)(1.3,-3.57)
\begin{scriptsize}
\psdots[dotstyle=*,linecolor=blue](13.56,4.16)
\psdots[dotstyle=*,linecolor=blue](5.47,2.36)
\psdots[dotstyle=*,linecolor=blue](6.87,2.37)
\psdots[dotstyle=*,linecolor=blue](2.66,-2.43)
\psdots[dotstyle=*,linecolor=blue](2.7,-3.55)
\psdots[dotstyle=*,linecolor=darkgray](1.26,-2.45)
\psdots[dotstyle=*,linecolor=darkgray](9.35,-0.64)
\psdots[dotstyle=*,linecolor=darkgray](1.3,-3.57)
\psdots[dotstyle=*,linecolor=darkgray](10.75,-0.62)
\psdots[dotstyle=*,linecolor=darkgray](5.47,2.36)
\psdots[dotstyle=*,linecolor=darkgray](14.96,4.18)
\psdots[dotstyle=*,linecolor=darkgray](9.39,-1.76)
\psdots[dotstyle=*,linecolor=darkgray](10.79,-1.74)
\psdots[dotstyle=*,linecolor=darkgray](13.59,3.05)
\psdots[dotstyle=*,linecolor=darkgray](14.99,3.07)
\psdots[dotstyle=*,linecolor=darkgray](5.5,1.24)
\psdots[dotstyle=*,linecolor=darkgray](6.9,1.26)
\end{scriptsize}
\end{pspicture*}
\end{center}

\nin For any finite subset $N \subset \N$, let
$N':= \{ k' \mid k \in N \}$ be another copy of $N$, and let
\begin{equation}
\cP(\ol N) := \cP(N \cup N') 
\end{equation} 
be the $2n$-hypercube defined by $\ol N:= N \cup N'$.
We distinguish several types of vertices, edges and faces:

\begin{definition}\label{def:vertextype}
Let $\alpha \subset N$, with preceding notation. A vertex of the form $(\alpha,\ol N)$ is called an {\em $N$-vertex}, 
a vertex of the form $(\alpha' , \ol N)$ an {\em $N'$-vertex}, and a vertex of the form
$(\ol \alpha,\ol N)$ is called a {\em saturated vertex}. A vertex that is of neither of these forms will be called {\em generic}.
\end{definition}

\begin{definition}\label{def:edgetype}
An edge $(\beta,\alpha) \in \cP(\ol N)$ will be called {\em of first kind} if the unique element
of $\alpha \setminus \beta$ belongs to $N$, and {\em of second kind} if it belongs to $N'$.
We distinguish three kinds of faces, namely faces such that

\ssk
(a) all of its edges are of first kind,

(b) all of its edges are of second kind,

(c) it has two edges of the first and two of the second kind.
\end{definition}

\nin
In the preceding illustration of a tesseract ($N = \sftwo = \{ 1,2 \}$),  edges of the first kind are horizontal or vertical, and edges of the
second kind are ``diagonal'' (grey). Vertices are labelled with the same convention as before. 
There are 16 edges of first kind and 16 edges of second kind, and
4 faces of type (a), 4 of type (b), and 16 of type (c). 
Moreover, we have indicated in grey the face formed by the $\sftwo$-vertices and the face formed by the $\sftwo'$-vertices.

\section{$n$-fold categories and groupoids}\label{App:nfoldCats}\label{app:nfoldcats}

\subsection{The concrete category of small $n$-fold cats}
As in Appendices A and C of Part I, the term {\em small (double) category} is abbreviated by ``small (double) cat'', and
{\em concrete category} by ccat. Ccats  are denoted by underlined roman letters. In this appendix, we
define {\em (strict) small $n$-fold categories} by induction (following the pattern from general category theory, see, e.g.,
the corresponding 
\href{http://ncatlab.org/nlab/show/n-fold+category}{pages in the $n$-lab}), 
and then establish a more direct and algebraic description (Theorem \ref{th:nfoldcat}):

\begin{definition}
For $n=1$, a small $1$-fold cat is a small cat, and this defines the concrete category $\ul{\rm Cat}_1$ of small cats. 
A small $n$-fold cat is a small cat $C$ of type $\ul{\rm Cat}_{n-1}$, and a morphism of small $n$-fold cats
is a functor $f:C \to C'$ which is also a morphism with respect to the $\ul{\rm Cat}_{n-1}$-structures on $C$ and $C'$.
This defines a concrete category $\ul{\rm Cat}_n$.
In the same way, the concrete category $\ul{\rm Goid}_n$ {\em of $n$-fold groupoids} is defined.
\end{definition}

\begin{theorem}\label{th:nfoldcat}
A small $n$-fold cat $C$ is equivalent to the following data: 
\begin{enumerate}
\item
a family of sets $(C_\alpha)_{\alpha\in \cP(\sfn)}$, indexed by vertices of the natural hypercube $\cP(\sfn)$,
\item
a structure of small cat $(\pi^e:C_\alpha \downdownarrows C_\beta, \ast_e,z_e)$, called {\em edge-cat}, 
for every (oriented) edge $e=(\beta,\alpha ; \sfn)$ of the hypercube,
\end{enumerate}
such that the following {\em face condition} holds:
for every face $(\gamma,\alpha;\sfn)$ of the hypercube, the edge categories form a small double cat:
\begin{equation*}
\begin{matrix}
\alpha & \to & \beta \cr \downarrow & & \downarrow \cr
\delta  & \to & \gamma  \end{matrix} \qquad \mbox{ gives a small double cat } \qquad
\begin{matrix}
C_\alpha & \rightrightarrows & C_\beta \cr \downdownarrows & & \downdownarrows \cr
C_{\delta} & \rightrightarrows  & C_\gamma  \end{matrix}
\end{equation*}
To shorten notation, we often write $(C_\alpha)_{\alpha \in \cP( \sfn)}$ to denote the whole structure.
{\em Morphisms of small $n$-fold cats} are equivalent to families of maps $f_{\alpha ;\sfn}: C_{\alpha;\sfn} \to C_{\alpha;\sfn}'$
such that, for each edge $(\beta,\alpha)$, the pair
$(f_\beta,f_\alpha)$, is a  morphism of small cats.
\end{theorem}

\begin{proof}
By induction on $n$.
For $n=1$, the equivalence is trivial (and the face condition is empty), and for  $n=2$, 
the equivalence is also trivial (since there is just one face, it says that, by definition, 
a $2$-fold cat is the same as a double cat. Note that in  Theorem 
C.1  from Part I an ``explicit'' description of double cats has been given). Assume now that the claim holds for some $n \in \N$.

\ssk
Assume $C$ is a small $n+1$-fold cat and let us show that it gives rise to a structure as described in the theorem. 
By assumption,  $C = C_0 \dot\cup C_1$, and $C_0,C_1$ are small $n$-fold cats. 
By induction, 
$C_0 = (C_\alpha)_{\alpha \in \cP(\mathsf{n})}$ and
$C_1 = (D_\alpha)_{\alpha \in \cP(\mathsf{n})}$ are families of sets, projections and products having the properties 
listed in the theorem.
Moreover, $\pi_\sigma : C_1 \rightrightarrows C_0$ defines two morphisms of $n$-fold cats, hence, again by induction,
is given by two families $(\pi^{\alpha}_\sigma : D_\alpha \rightrightarrows C_\alpha)_{\alpha \in \cP(\mathsf{n})}$
which for each $\alpha$ define a morphism of small cats. 
Letting for each ${\alpha \in \cP(\mathsf{n})}$,
$$
C_{\alpha \cup \{ n+1 \}} := D_\alpha, \qquad
\pi^{\alpha,\alpha \cup \{ n +1 \};\mathsf{n+1}} := \pi^\alpha :  D_\alpha = C_{\alpha \cup \{ n +1 \}}  \rightrightarrows C_\alpha,
$$
and by distinguishing several cases (the three cases of ``old'', ``copies of old'' and ``new'' edges, and similar for faces,
see section  \ref{subsec:talking}),
one checks that the data $(C_\alpha)_{\alpha \in \cP(\mathsf{n+1})}$ satisfy the properties from the theorem.

\ssk
Conversely, assume data $(C_\alpha)_{\alpha \in \cP(\mathsf{n+1})}$ with properties as in the theorem are given.
Then $C_0:=  (C_\alpha)_{\alpha \in \cP(\mathsf{n})}$ and
$C_1 := (C_\alpha)_{\alpha \in \cP(\mathsf{n+1})\setminus \cP(\mathsf{n})}$ 
both satisfy the induction hypothesis, hence define small $n$-fold cats.
The family of projections
$\pi^{\alpha,\alpha \cup \{ n+1 \};\mathsf{n+1}}$ is, by induction, a morphism between these $n$-fold cats, and so is the family
of unit sections. 
To conclude that $(C_0,C_1,\ast)$ is a small cat of type $\ul{\rm Cat}_n$, it only remains to show that  $\ast$ is a morphism
of small $n$-fold cats.
By induction, this is equivalent to showing that $\ast$ defines, on the level of each ``old edge'' $(\gamma,\beta)$ 
coming from  $\cP(\sfn)$, 
a morphism of small cats. This, in turn, is equivalent to saying that the ``new face'' 
$(\gamma, \beta \cup \{ n+1 \})$ defines a small double cat. But this is ensured by the face condition, whence the claim.
\end{proof}

\begin{example}
If $C_1,\ldots,C_n$ are small cats (resp.\ groupoids), then the {\em external pro\-duct}
$C_1 \square \ldots \square C_n$ is an $n$-fold cat (resp.\ $n$-fold groupoid)
(see \cite{FP10}, def.\ 2.9).
\end{example}

\begin{theorem}
For a small $n$-fold cat $C = (C_\alpha)_{\alpha \in \cP(\sfn)}$, the following are equivalent:
\begin{enumerate}
\item
$C$ is an $n$-fold groupoid,
\item
for each edge $(\beta,\alpha;\sfn)$ of $\cP(\sfn)$, the edge cat $(C_\beta,C_\alpha)$ is a groupoid.
\end{enumerate}
\end{theorem}

\begin{proof}
Straightforward by induction. 
\end{proof}

See also \cite{FP10} for a  concise proof of the essentially same results. 
Note that the description given in the theorems does not refer to the total order of $\cP(\sfn)$, but only to its partial order.
Therefore every permutation of $\sfn$ defines again an $n$-fold cat:

\begin{definition} 
Let $\tau \in S_n$ be a permutation of $\sfn$. For a small $n$-fold cat $C=(C_\alpha)_{\alpha \in \cP(\sfn)}$, the
{\em $\tau$-transposed $n$-fold cat} is given by $C^\tau:=(C_{\tau(\alpha)})_{\alpha \in \cP(\sfn)}$.
The $n$-fold cat $(C_\alpha)_{\alpha \in \cP(\sfn)}$ is called {\em edge-symmetric} if, for each $\tau \in S_n$, there is
an isomorphism of $n$-fold cats $f_\tau : C \to C^\tau$ such that
$f_\tau \circ f_\nu = f_{\tau \nu}$ and $f_{\id}=\id_{C}$. In other words, we have an action of $S_n$ by isomorphisms.\footnote{
In work of Brown and Higgins \cite{BH81}, and in \cite{BHS11}, it is always assumed that $n$-fold cats and groupoids
are edge symmetric, and notation is then chosen such that, whenever $\alpha$ and $\alpha'$ are conjugate under $S_n$
(i.e., have same cardinality), $C_\alpha$ and $C_{\alpha'}$ are denoted by the same symbols.}
\end{definition}

\begin{definition}
Let $C = (C_0,C_1,\ast)$ be  a small $n$-fold cat.
Then the opposite cat $C^{opp}$ is again a small $n$-fold cat. Iterating this procedure, 
for each subset $\gamma \subset \sfn$ we get the
{\em $\gamma$-opposite $n$-fold cat} $C^{\gamma{\rm-opp}}$ defined by taking at level $i$ the opposite cat structure if $i \in \gamma$
and the original one else. 
\end{definition}

\begin{definition}\label{def:topdown}
In a small $n$-fold cat,  the {\em $n$-fold source map} is defined by
$$
\xi : = \pi^{\emptyset,\sfone}_0 \circ \pi^{\sfone,\sftwo}_0 \ldots \circ \pi^{\mathsf{n-1},
\sfn}_0  :
 C_\sfn \to C_\emptyset .
$$
More generally, for any subset $\gamma \subset \sfn$, we define the
{\em $\gamma$-top-down projection}
$$
\xi_\gamma : = \pi^{\emptyset,\sfone}_{k_1} \circ \pi^{\sfone,\sftwo}_{k_2} \ldots \circ \pi^{\mathsf{n-1},\sfn}_{k_n}  
 : C_\sfn \to C_\emptyset
$$
where $k_i = 1$ (target) if $i\in \gamma$ and
$k_i=0$ (source) if $i \notin \gamma$. In other words, this is the $n$-fold source map of $C^{\gamma{\rm-opp}}$.
\end{definition}

\subsection{Two-typed $2n$-fold cats}

\begin{lemma}
Let $k \in \{ 1,\ldots, n-1\}$.
The following are equivalent:
\begin{enumerate}
\item
$C$ is a small $n$-fold cat,
\item
$C$ is a small $k$-fold cat of type $\ul{\rm Cat}_{n-k}$.
\end{enumerate}
\end{lemma}

\begin{proof}
By definition, (2) means that $C$ is a small $k$-fold cat and a small $n-k$-fold cat such
that all structure maps of the $k$-fold cat are morphisms of the $n-k$-fold cat structure.
Applying the description from the theorem, this amounts to (1). 
\end{proof}

\nin
Note, however, that (2) defines a kind of additional structure on $C$, when the value of $k$ is considered to be fixed. 

\begin{definition}
A {\em two-typed $2n$-fold cat} is a small $2n$-fold cat, together with a fixed structure of small $n$-fold cat of type $\ul{\rm Cat}_n$.
As index set, we then use the two-typed hypercupe $\cP(\ol \sfn)$, and we often denote projections and composition laws
corresponding to edges of the second
kind by the symbols  $\partial$ and $\bullet$, whereas for those corresponding to edges of the first kind we use $\pi$ and $\ast$.
\end{definition}

\subsection{Cat-rules}\label{sec:catconstruction}
The $n$-fold cats considered in this work arise in the following way:

\begin{definition}
Fix some  ccat $\ul{\rm C}$.
We assume that one can form pullbacks in $\ul{\rm C}$ (cf.\ Appendix \ref{app:pullback}).
 A {\em cat-rule $\sfQ$ on $\ul{\rm C}$} 
 is a pullback-preserving functor from $\ul{\rm C}$ to \ul{\rm C-Cat}. That is, we have 
a functor $\sfQ = \sfQ^\sfone = \sfQ^\sett{1}$ from $\ul{\rm C}$ to itself, 
such that each set
$\sfQ^\sfone U$ carries a structure of small cat of type $\ul{\rm C}$ (with object set
$\sfQ^{\emptyset;\sfone} U$, morphism set $\sfQ^{\sfone;\sfone} U$ and composition $\ast_1$),  and each
  $\sfQ  f:\sfQ U \to \sfQ U'$ is a morphism of small cats. 
Moreover, we require that   
 $\sfQ$ be compatible with (set-theoretic) pullbacks, i.e.,
\begin{equation}\label{eqn:pullbackQ}
\sfQ (A \times_B C) = \sfQ A \times_{\sfQ B} \sfQ C .
\end{equation}
\end{definition}
  
\begin{theorem}\label{th:catrules}
Assume $\sfQ$ is a cat-rule on $\ul{\rm C}$. Applying  a second copy $\sfQ^\sett{2}$ of $\sfQ^\sett{1}$ to the data of the small cat
$\sfQ^\sfone U$, we get a small double cat
\begin{equation*}
\sfQ^\sett{1,2} U := \sfQ^\sett{2} (\sfQ^\sett{1} U) ,
\end{equation*}
and iterating this procedure, we get a small $n$-fold cat
$\sfQ^\sfn U = \sfQ^\sett{n} (\sfQ^\sett{1,\ldots,n-1} U)$.
\end{theorem}

\begin{proof}
Since $\sfQ^\sett{2}$ is a functor, applying it to the small cat structure with composition $\ast_1$ gives another small cat
with ``derived'' composition $\sfQ^\sett{2}(\ast_1)$ and  ``derived'' projections and unit section. 
To make this statement formally correct, we have, however, to pay attention to the domains of definition of composition laws:
composition $\ast_1$ is defined on a domain $D = A \times_B C \subset A \times C$ with 
$A=C=\sfQ^{\sfone;\sfone} U$ and $B=\sfQ^{\emptyset;\sfone} U$, and thus
$\sfQ^\sett{2} (\ast_1)$ is, {\it a priori}, defined on $\sfQ^\sett{2} D$.
But the very definition of a small cat requires that this product be defined on the set 
$\sfQ^\sett{2} A \times_{\sfQ^\sett{2} B} \sfQ^\sett{2} C$ (pullback using projections of
 $\sfQ^\sett{2} (\sfQ^\sett{1} U)$). Therefore, a sufficient condition for the construction to work is that 
  $\sfQ$ be compatible with pullbacks, i.e., (\ref{eqn:pullbackQ}).
 
On the other hand,   $\sfQ^\sett{2} (\sfQ^\sett{1} U)$ arises from applying  $\sfQ^\sett{2}$ to a set, hence  is a small cat with composition $\ast_2$. Since ``derived'' maps $\sfQ^\sett{2} f$ are compatible with such structure,
all structure maps defined above are morphisms of the small cat with composition $\ast_2$, and so
$\sfQ^\sett{1,2}U$ is a small doublecat. 
Applying another copy $\sfQ^\sett{3}$ to this structure, we get a small triplecat, and so on. 
\end{proof}

\begin{example}
Our examples of cat-rules are:
\begin{enumerate}
\item
the {\em pair groupoid rule} $\pG$, defined on $\ul{\rm C} = \ul{\rm Set}$ (Appendix \ref{app:pg}),
\item
the {\em scaled action groupoid rule} $\sfA$, defined on the ccat of $S$-modules for a commutative monoid $S$
(Appendix \ref{app:scaledaction}),
\item
the {\em derived groupoid rule} $\sfG$ (main object of this work), defined on the ccat of linear sets, with morphisms initially the
polynomial laws (but then extended to all $C^\infty$-laws), and pullbacks as in Appendix \ref{app:pullback},
\item
the {\em finite part derived groupoid rule} $\sfGfi$, defined on the ccat of linear sets, with morphisms arbitrary set-maps,
\item 
the {\em symmetric derived groupoid rule} $\sfGsy$, with morphisms as for $\sfG$,
\item
for each $t \in \K$, the rule $\sfG_t$ (in particular, for $t=0$, the {\em tangent functor} $\sfT$).
\end{enumerate}
\end{example}
  
\nin  For all of these examples, we check Property (\ref{eqn:pullbackQ}) ``by hand''
(see Appendix \ref{app:pullback} for the most important case $\sfQ = \sfG$). 
 
\begin{definition}
An {\em $n$-fold cat rule} on a ccat $\ul{\rm C}$ is given by functors, for each subset $N \subset \N$ of cardinality $n$,
$\sfQ^N$, from $\ul{\rm C}$ to the ccat $\ul{\rm C-Cat}_N$ of small $n$-fold cats of type ${\rm C}$, and preserving pullbacks.
\end{definition}

Thus any cat-rule gives rise, by iteration, to an $n$-fold cat rule.
But it is also possible to compose different cat-rules with each other:

\begin{theorem}
Assume $\sfQ$ and $\sfP$ are cat-rules defined on a common ccat $\ul{\rm C}$.
Then $\sfQ \circ \sfP$ and $\sfP \circ \sfQ$ are doublecat rules on $\ul{\rm C}$ (different from each other, in general).
\end{theorem}

\begin{proof}
Same arguments as in the proof of Theorem \ref{th:catrules}.
\end{proof}

\begin{example}
The cat rules 
$\sfQ = \sfG$ and $\sfP = \sfA$ (for $S=K=\K$) cannot be composed  since $\sfG$ does not preserve the required equivariance
properties (we cannot define the common ccat  $\ul{\rm C}$).
On the other hand, $\sfG^{\ol \sfone}$ is a doublecat rule, which in effect does what a hypothetical composition of $\sfG$ and $\sfA$ 
should do. 
\end{example}

\section{The $n$-fold pair groupoid}\label{app:pg}

Let $M$ be a set. Recall that the {\em pair groupoid} $\pG^\sfone M$ of $M$ is given by the two projections
$M \times M \downdownarrows M$, unit section the diagonal map $M \to M\times M$ and product
$(x,y) \ast (y,z) = (x,z)$ and inverse $(x,y)^{-1}=(y,x)$.
Every set-map $f:M \to M'$ induces a groupoid morphism
$\pG^\sfone f = (f \times f,f)$, and  $\pG^\sfone$ is a cat-rule (even, a groupoid-rule).
Let $\pG^\sett{k}$ be a copy of $\pG^\sfone$, ``of $k$-th generation''.
It follows that $\pG^\sftwo M:= \pG^\sett{2} (\pG^\sfone M)$ is a double groupoid (see also \cite{Ko87}), etc.:
\begin{equation}
\pG^\sfn M := \pG^\sett{n} (\pG^\sett{1,\ldots,n-1} M )
\end{equation}
is an $n$-fold groupoid, called the {\em $n$-fold pair groupoid}. 
Let us describe it in terms of hypercubes, as in Theorem \ref{th:nfoldcat}:

\begin{theorem}
Let $M$ be a set.  For every $\beta \in \cP(\sfn)$, let $M_\beta$ be a formal copy of $M$, 
and define the vertex sets, for $\alpha \in \cP(N)$,
\begin{equation}
\pG^{\alpha,\sfn} M:=  \prod_{\beta \in \cP(\alpha)} \,  M_\beta \cong M^{2^{\ell_\alpha}} .
\end{equation}
Let $(\beta,\alpha)=(\beta,\beta \cup \{ i \})$ be an edge, where $i \notin \beta$.
The two projections corresponding to this edge are defined by
\begin{equation}
\pi_0^{\beta,\alpha} ((x_\gamma)_{\gamma \subset \alpha}) :=
(x_\gamma)_{\gamma \subset \beta}, \qquad
\pi_1^{\beta,\alpha} ((x_{\gamma} )_{\gamma \subset \alpha}) :=
(x_{\gamma \cup \{ i\} })_{\gamma \subset \beta}
\end{equation}
the zero sections are the ``diagonal imbedding''
$(x_\gamma) \mapsto (x_\gamma,x_\gamma)$,
and the product corresponding to this edge is defined for $x,y$ such that $x_{\gamma \cup \{ i\}} = y_\gamma$
(for $\gamma \subset \beta$), by
\begin{equation}
(x_\gamma) \ast_{\alpha,\beta} (y_\gamma)=(z_\gamma), \quad
z_\gamma = \Bigl\{
\begin{matrix}
x_\gamma & \rm{if} & i \notin \gamma
\cr
y_\gamma & \rm{if} & i \in \gamma .
\end{matrix}
\end{equation}
These data define an edge-symmetric $n$-fold groupoid, naturally isomorphic to the inductively defined $\pG^\sfn M$.
\end{theorem}

\begin{proof}
For $n=1$, 
$\pi_0^{\emptyset,\sfone}(x_\emptyset,x_1)= x_\emptyset$,
$\pi_1^{\emptyset,\sfone}(x_\emptyset,x_1)=x_1$,
$(x_\emptyset,x_1)\ast (y_\emptyset,y_1)=(x_\emptyset,y_1)$
whenever $y_\emptyset = x_1$, describing the pair groupoid. 
For $n=2$, we have the following source and target projections, 
$$
\begin{matrix}
M^4 & \rightrightarrows & M^2 \cr \downdownarrows & & \downdownarrows \cr
M^2 & \rightrightarrows & M 
\end{matrix},
\qquad 
\begin{matrix}
(x_0,x_1,x_2,x_{12} )& \mapsto & (x_0,x_2) \cr \downarrow & \pi_0 & \downarrow \cr
(x_0,x_1) & \mapsto & x_0 
\end{matrix}
\qquad
\begin{matrix}
(x_0,x_1,x_2,x_{12}) & \mapsto & (x_1,x_{12}) \cr \downarrow & \pi_1 & \downarrow \cr
(x_2,x_{12}) & \mapsto & x_{12}
\end{matrix}
$$
which match precisely the data of $\pG^\sett{2} \pG^\sfone M$. 
Similarly for the general induction step. 
Edge-symmetry is directly seen for $n=2$ from the formulae above; more formally,
whenever $\alpha$ and $\alpha'$ have same cardinality and $\tau:\alpha \to \alpha'$ is a bijection,
there is an induced bijection $f_\tau: M_\alpha \to M_{\alpha'}$, which together assemble to  define an isomorphism.
\end{proof}

Note also that
the $2^n$ top-down projections of $\pG^\sfn M$ are precisely the $2^n$ projections of $M^{2^n}$ to $M$, and that every
 map $f:M \to M'$ induces a morphism of $n$-fold groupoids
$f^\sfn : \pG^\sfn(M) \to \pG^\sfn(M')$, such that the symbol $\pG^\sfn$ can be seen as an
$n$-fold groupoid rule on $\ul{\rm Set}$.

\begin{theorem}
For every small $n$-fold cat $C = (C_\alpha)_{\alpha \in \cP(\sfn)}$, the {\em iterated anchor map}
$$
C \to \pG^\sfn (C_\emptyset)
$$
is morphism of small $n$-fold cats. On $C_\alpha$, this map is given by associating to $x \in C_\alpha$
the family $(\xi_\gamma(x))_{\gamma \in \cP(\alpha)}$ where $\xi_\gamma$ is the $\gamma$-top-down projection 
(Definition \ref{def:topdown}).
\end{theorem}

\begin{proof}
For $n=1$, this map is the usual anchor map $x \mapsto (\pi_0(x),\pi_1(x))$, which is a morphism $C \to \pG^\sfone (C_\emptyset)$.
For $n>1$, the claim follows by induction.
\end{proof}

\begin{definition}
An {\em $n$-fold equivalence relation on $M$} is a subgroupoid of $\pG^\sfn M$.
\end{definition}

\nin
For $n=1$, we get a usual equivalence relation, and for $n=2$, we get a 
{\em double equivalence relation} as defined in  \cite{Ko87}.


\section{The $n$-fold scaled action category}\label{app:scaledaction}

Let $S$ be a
 commutative monoid, acting (from the left) on a set $K$ (``the scales'', later we will take for $S$ and $K$ a copy of $\K$).
Assume $S$ acts (from the right) on a set $V$, and let $U \subset V$ a non-empty subset.
We define  a small cat $\sfA^\sfone U$ (see Part I, Lemma B.1) by sets
\begin{align}
\sfA^{\emptyset,\sfone} U &:=  U \times K \cr
\sfA^{\sfone,\sfone} U &:= \{ (v,s,t) \in V \times S \times K  \mid \, v \in U, vs  \in U \}  , 
\end{align}
with injection $(v,t) \mapsto (v,1,t)$ and
and projections  $\partial: \sfA^{\sfone,\sfone} U   \downdownarrows \sfA^{\emptyset,\sfone} U$ given by
\begin{equation}
\partial_1(v,s,t)=(vs,t), \qquad
\partial_0(v,s,t)=(v,st)
\end{equation}
and composition law
$(v',s',t') \bullet (v,s,t)= (v,ss',t')$. 

\begin{lemma}With notation as above,
\begin{enumerate}
\item
 every $S$-equivariant map  $f:V \to V'$ defines a morphism
$$
\sfA^\sfone f =(f \times \id_S \times \id_K,f \times \id_K),
$$
\item
let $S$ act on $V \times S \times K$ and on $V \times K$ by:
$$
(v,s,t) . \lambda := (v \lambda , s,t), \qquad
(v,t). \lambda := (v \lambda,t) ,
$$
then both projections $\partial_\sigma$ and  composition $\bullet$ are  equivariant maps.
\end{enumerate}
\end{lemma}

\begin{proof} Direct check. As to (1),
equivariance of $f$ is needed to ensure that
$\partial_1 \circ \sfA^\sfone f = (f \times \id) \circ \partial_1$,
and as to (2), equivariance of  $\partial_1$, amounts to
 $v\lambda s = vs \lambda$ for all $s$ and $\lambda$, so here commutativity of $S$ enters.
\end{proof}

\begin{theorem}\label{th:SA}
Let $\sfA^\sett{k}$ be a formal copy of the functor $\sfA^\sfone$ defined above. Then, for any non-empty
subset $U \subset V$, we get a small  edge symmetric $n$-fold cat, defined by induction
$$
\sfA^\sfn U := \sfA^\sett{n} (\sfA^{\mathsf{n-1}} U) =
\sfA^\sett{n} \ldots \sfA^\sett{1} (U).
$$
Its vertex sets, edge projections and injections and composition laws are described as follows:
For any $i \in \N$ let $S_i$ be a formal copy of $S$ and $K_i$ one of $K$.
For $\alpha = \{ \alpha_1,\ldots, \alpha_\ell \} \in \cP(\sfn)$  the vertex sets are
\begin{align*}
\sfA^{\alpha,\sfn} U &=\bigl\{ (x,\sss,\ttt) \in 
 U \times S_{\alpha_1}\times \ldots \times S_{\alpha_\ell} \times K_1 \times \ldots \times K_n \mid
 \cr
 & \qquad \qquad \forall \beta \in \cP(\alpha) : x s_{\beta_1} \cdots s_{\beta_\ell} \in U \bigr\} ,
\end{align*}
and for an edge $(\beta,\alpha)$ with $\alpha = \beta \cup \{ i \}$, the edge projections are
$$
\partial_1 (x,\sss,\ttt ) = (x s_i , (s_k)_{k\in \beta}, \ttt), \qquad
\partial_0(x,\sss,\ttt) = (x, (s_k)_{k\in \beta}, s_i t_i, (t_k)_{k\in \beta}) ;
$$
in other words, each $s_i$ acts on $V$ but
 only on the $i$-th component of $K^n$. The zero section is the natural inclusion, and the composition law is
$$
(x,\sss',\ttt') \bullet (x,\sss,\ttt) = (x, s_i s_i', (s_j)_{j\not= i},  \ttt').
$$
Moreover, any map $f:U \to U'$
satisfying $f(vs)=f(v)s$ whenever $v \in U$, $s \in S$ and $vs \in U$
 gives rise to a morphism $\sfA^\sfn f : \sfA^\sfn U \to \sfA^\sfn U'$.
In particular, there is a canonical morphism $\sfA^\sfn U \to \sfA^\sfn 0$, where
$0$ is the zero $S$-module. 
\end{theorem}

\begin{proof}
From the preceding lemma, it follows that $\sfA$ is a cat-rule (by direct check, $\sfA$ is compatible with pullbacks), 
and hence by iterating we get an $n$-fold cat. 
The explicit formulae for vertex sets, edge projections and composition laws are proved by induction. 
These formulae show an edge-symmetric structure, when we let act the symmetric group $S_n$ on the indices and
define the maps on the level of vertex sets by the corresponding identification maps. 
For instance, when  $n=2$ and $U=V$, we have a diagram of projections
$$
\begin{matrix}
V \times S_1 \times S_2 \times K_1 \times K_2 & \rightrightarrows  & V \times S_1 \times K_1 \times K_2 \cr
\downdownarrows & & \downdownarrows \cr
V \times S_2 \times K_1 \times K_2 & \rightrightarrows & V \times K_1 \times K_2 ,
\end{matrix}
$$
having the natural symmetry properties. 
\end{proof}

\begin{example}
If $S$ acts trivially on $V$, then $\sfA^\sfn U = U \times \sfA^\sfn 0$ is just a direct product of $U$ with the iterated ``usual'' 
$n$-fold action cat of $S$ on $K$. 
\end{example}


\section{Pullbacks in first order calculus}\label{app:pullback}

In general category theory, {\em pullbacks} are defined by a universal property. In the concrete cat $\ul{\rm Set}$, 
they are given by the following construction (see \cite{MM92}, I.2), which equally applies to the concrete cat of linear sets with
set-maps: 

\begin{definition}\label{def:pullback1}
Assume  $A \subset V_1$, $B \subset V_2$, $C \subset W$ are linear sets and
 $f:A \to C$ and $g:B \to C$ maps. 
The {\em pullback} (in the concrete category of linear sets with set maps),
 $ (A \times_C B, A,B,p_A,p_B,f,g)$, of these data is the linear set defined by 
\begin{align*}
P:=A \times_C B & := \{ (a,b) \in A \times B \mid f(a)=g(b) \}  \subset V_1 \times V_2 ,
\end{align*} 
{\em together} with its two projections $p_A:P \to A$ and $p_B:P \to B$ fitting into the commutative diagram of mappings
\begin{align}
\begin{matrix}
P & \to & B \cr \downarrow & & \downarrow \cr A & \to & C . \end{matrix} 
\end{align}
In the following, we distinguish the {\em set} $P$ and the whole structure of pullback, denoted by $A \times_C B$.
Note that,
if $C=0$ is a point, then $P$ is the usual cartesian product $A \times B$ of sets. 
\end{definition}

However, this definition is not appropriate for the concrete  cat of linear sets with $C^1$-laws. The reason for this is that base maps need not determine the whole law. 
Morally, they do if the domain of definition is ``open''; but, for instance, diagonals $\Delta_V$ in vector spaces $V$ are typically not open in
$V \times V$,
and thus   the set
$(\Delta_V)^\sett{1}$ is ``too big'': we want it to be equal to  $\Delta_{V^\sett{1}}$. In other words, we have ``to derive maps first and then take
set-theoretic pullbacks'', and not the other way round:

\begin{definition}\label{def:pullback}
Recall from Definition \ref{def:ccatC1} the ccat of linear sets with $C^1$-laws.
With notation as above, 
 we {\em define the pullback in the concrete cat of linear sets with $C^1$-laws} to be the set-theoretic pullback 
of $\sfG^\sfone A$ and $\sfG^\sfone C$   via $\sfG^\sfone f$ and $\sfG^\sfone g$, i.e., it is given by the base set
$A \times_C B$ and the extended sets
\begin{equation}\label{eqn:pb!}
\sfG^{\sfone}  (A \times_C B)  := \sfG^{\sfone} A \times_{\sfG^{\sfone} C} \sfG^{\sfone} B . 
\end{equation}
Similar definitions are given for $\sfG^{\ol \sfone}$ and $\sfG^\settt{k}$. In particular,  with $\K_1 := 0^\sett{1}$,
$$
\sfG^{\sfone}  (A \times B) = \sfG^{\sfone} A \times_{\K_1} \sfG^{\sfone} B, \qquad 
\sfG^{\ol \sfone}  (A \times B) = \sfG^{\ol\sfone} A \times_{\K_1 \times \K_1} \sfG^{\ol \sfone} B.
$$
\end{definition} 

\nin
The following theorem ensures that the pullbacks thus defined are again groupoids (resp.\ small double cats).

\begin{theorem}\label{la:pullback}
With notation as in Definition \ref{def:pullback1}, assume  $\bff,\bfg$ are $C^1$-laws with base maps $f,g$. Then the
following pair of linear sets is a subgroupoid of the groupoid $\sfG^\sfone P = (P\times \K, \sfG^{\sfone;\sfone} P)$:
$$
 \sfG^{\sfone} A \times_{\sfG^{\sfone} C} \sfG^{\sfone} B =
 \bigl( \sfG^{\emptyset;\sfone} A \times_{\sfG^{\emptyset;\sfone} C} \sfG^{\emptyset;\sfone} B,
 \sfG^{\sfone;\sfone} A \times_{\sfG^{\sfone;\sfone} C} \sfG^{\sfone;\sfone} B \bigr)
$$
A similar statement holds for $\sfG^\sfone$ replaced by $\sfG^{\ol \sfone}$.  
\end{theorem}

\begin{proof}
On the level of object sets, both groupoids from the claim coincide, by using the natural identification
$(A \times \K) \times_{(C \times \K)} (B \times \K) = (A \times_C B) \times \K$. 
On the level of morphism sets, 
let us show that $(\sfG^{\sfone;\sfone} A \times_{\sfG^{\sfone;\sfone} C} \sfG^{\sfone;\sfone} B ) \subset \sfG^{\sfone;\sfone} P$. 
On the one hand,
\begin{align*}
P^\sett{1} &=
\Bigsetof{ ((a,b),(u,v),t))} {\begin{array}{c} a\in A, b \in B, u \in V_1, v \in V_2 ,t \in \K :
 a+ut \in A, b+tv \in B, \\ f(a) = g(b), \quad
f(a+ut)=g(b+vt) 
\end{array}}
\cr
& =
\Bigsetof{ ((a,b),(u,v),t))} {\begin{array}{c} a\in A, b \in B, u \in V_1, v \in V_2,t \in \K : 
a+ut \in A, b+vt \in B, \\ f(a) = g(b), \quad
f^{[1]}(a,u,t) \, t  =  g^{[1]}(b,v,t) \,  t  
\end{array}} .
\end{align*}
(For the last equality, note that $f(a+ut)=f(a) +  f^{[1]}(a,u,t) \, t$. Note also that the set $P^\sett{1}$ depends only on the base
maps of $\bff$ and $\bfg$.) 
On the other hand, using the maps $f^\sett{1}:A^\sett{1}\to C^\sett{1}$ and $g^\sett{1}:B^\sett{1} \to C^\sett{1}$:
\begin{align*}
Q & := A^\sett{1} \times_{C^\sett{1}} B^\sett{1}
\cr &  =
\Bigsetof{ ((a,u,t),(b,v,s))} {\begin{array}{c}  (a,u,t) \in A^\sett{1}, \, (b,v,s) \in B^\sett{1} :
 \\  f^\sett{1}(a,u,t) = g^\sett{1}(b,v,s) 
\end{array}}
\cr
& = \Bigsetof{ ((a,u,t),(b,v,t))} {\begin{array}{c} a\in A, b \in B, u \in V_1, v \in V_2, t\in \K: a+tu \in A, b+tv \in B,   \\ 
f(a) = g(b),  \quad f^{[1]}(a,u,t) = g^{[1]}(b,v,t) 
\end{array}} 
\end{align*}
(and this set set does not only depend on the base maps).
Since, obviously, 
$f^{[1]}(a,u,t) = g^{[1]}(b,v,t) \Rightarrow f^{[1]}(a,u,t) \, t= g^{[1]}(b,v,t) \, t$,
we have $Q \subset P^\sett{1}$.
Now let us show that $Q$ is stable under the law $\ast$ of $P^\sett{1}$.
Write $(a,u;b,v;t)$ for the element $((a,u,t),(b,v,t)) \in Q$, so
$f^{[1]}(a,u,t) = g^{[1]}(a,u,t)$, etc.; then, for $a'=a+ut$, $b'=b+vt$, 
\begin{align*}
(a',u';b',v';t) \ast (a,u;b,v;t) &=
(a,u'+u;b,v'+v;t)  
\end{align*}
satisfies, since $f^\sett{1}$ and $g^\sett{1}$ are groupoid morphisms,
\begin{align*}
f^{[1]}(a,u'+u,t) & = f^{[1]}(a',u',t) + f^{[1]}(a,u,t) \cr
 &=  g^{[1]}(a',u',t) + g^{[1]}(a,u,t) = g^{[1]}(a,u'+u,t) ,
\end{align*}
and similarly for $(b,v'+v',t)$. The description of $Q$ given above shows now that $(a',u';b',v';t) \ast (a,u;b,v;t) \in Q$.
Obviously, $Q$ contains the units and is stable under inversion, hence  $Q$ is a subgroupoid. 
This proves the first claim, and the final claim follows from this by a straightforward computation, using that, by definition,
$U^\settt{1} = \{ (x,v,s,t) \in V^2 \times \K^2 \mid \, (x,v,st) \in U^\sett{1} \}$.
\end{proof}

\begin{example} In general, the inclusion is strict. However, if
$C =0$ is a point, so $f:A \to 0$ and $g:B \to 0$ are constant and
$A \times_C B = A \times B$, we get  with $C^\sett{1} \cong \K$,
$$
(A \times B)^\sett{1} = (A \times_C B)^\sett{1} \supset A^\sett{1} \times_\K B^\sett{1},
$$
and in this case the inclusion is an equality (see  Theorem 2.10 in  
Part I). 
One may interprete this by saying that, ``when $C=0$, the diagonal is  open in $C \times C$''.
\end{example}

\end{document}